\documentclass[11pt]{article}
\usepackage{amssymb}
\usepackage{graphicx}
\usepackage{amsmath}
\usepackage{mathrsfs}

\newenvironment{proof}[1][Proof]{\textbf{#1.}
}{\ \rule{0.5em}{0.5em}}

\def\inbar{\,\vrule height1.5ex width.4pt depth0pt}
\def\inbar{\,\vrule height1.5ex width.4pt depth0pt}
\def\IB{\relax{\rm I\kern-.18em B}}
\def\IC{\relax\hbox{$\inbar\kern-.3em{\rm C}$}}
\def\ID{\relax{\rm I\kern-.18em D}}
\def\IE{\relax{\rm I\kern-.18em E}}
\def\IF{\relax{\rm I\kern-.18em F}}
\def\IG{\relax\hbox{$\inbar\kern-.3em{\rm G}$}}
\def\IH{\relax{\rm I\kern-.18em H}}
\def\II{\relax{\rm I\kern-.18em I}}
\def\IK{\relax{\rm I\kern-.18em K}}
\def\IL{\relax{\rm I\kern-.18em L}}
\def\IM{\relax{\rm I\kern-.18em M}}
\def\IN{\relax{\rm I\kern-.18em N}}
\def\IO{\relax\hbox{$\inbar\kern-.3em{\rm O}$}}
\def\IP{\relax{\rm I\kern-.18em P}}
\def\IQ{\relax\hbox{$\inbar\kern-.3em{\rm Q}$}}
\def\IR{\relax{\rm I\kern-.18em R}}
\font\cmss=cmss10 \font\cmsss=cmss10 at 7pt
\def\IZ{\relax\ifmmode\mathchoice
{\hbox{\cmss Z\kern-.4em Z}}{\hbox{\cmss Z\kern-.4em Z}}
{\lower.9pt\hbox{\cmsss Z\kern-.4em Z}}
{\lower1.2pt\hbox{\cmsss
Z\kern-.4em Z}}\else{\cmss Z\kern-.4em Z}\fi}
\def\IGa{\relax\hbox{${\rm I}\kern-.18em\Gamma$}}
\def\IPi{\relax\hbox{${\rm I}\kern-.18em\Pi$}}
\def\ITh{\relax\hbox{$\inbar\kern-.3em\Theta$}}
\def\IOm{\relax\hbox{$\inbar\kern-3.00pt\Omega$}}

\newtheorem{defi}{Definition}[section]
\newtheorem{theo}[defi]{Theorem}
\newtheorem{lemm}[defi]{Lemma}
\newtheorem{prop}[defi]{Proposition}

\newtheorem{coro}[defi]{Corollary}
\newtheorem{rema}[defi]{Remark}

\newcommand{\be}{\begin{enumerate}}
\newcommand{\ee}{\end{enumerate}}

\begin{document}
\title{Formulas for the Multiplicity of Graded Algebras}
\maketitle
\begin{center}
\author{Yu Xie}
\footnote{This paper is based on the author's Ph.D. thesis, written under the direction of Professor  Bernd Ulrich.
The author sincerely thanks Professor Ulrich for suggesting the problem and for advice and many helpful discussions.
}

Department of Mathematics,\\
The University of Notre Dame, South Bend, IN 46556
\\
E-mail: yxie@nd.edu
\end{center}




\begin{abstract}
Let $R$ be a standard graded Noetherian algebra over an Artinian
local ring. Motivated by the work of Achilles and Manaresi in
intersection theory, we first express the multiplicity of $R$ by
means of local $j$-multiplicities of various hyperplane sections.
When applied to  a homogeneous
 inclusion $A\subseteq B$ of standard graded Noetherian algebras
over an Artinian local ring, this formula yields the multiplicity of $A$
in terms of that of
$B$ and of local $j$-multiplicities of hyperplane sections along
${\rm Proj}\,(B)$. Our formulas can be used to find the
multiplicity of  special fiber rings and to obtain the degree of dual varieties
 for any hypersurface.
 In particular, it gives  a
generalization of Teissier's Pl\"{u}cker formula to  hypersurfaces
 with non-isolated singularities.
 Our work generalizes results by Simis,
Ulrich and Vasconcelos on homogeneous embeddings of graded
algebras.
\end{abstract}

\bigskip

\section{Introduction.}
\noindent

Let $A=A_0[A_1]\subseteq B=B_0[B_1]$ be a homogeneous inclusion of
standard graded Noetherian rings with $A_0=B_0$ and each Artinian local.
The goal of this work is to give a formula for the multiplicity of
$A$ in terms of the multiplicity   of $B$ and of local multiplicities along
${\rm Proj}\,(B)$.

One of the main applications of this formula will be to the computation of the multiplicity of special fiber rings.
Let $R$ be a standard graded Noetherian
algebra of dimension $d$ over a field $k$ and $I$  an ideal of $R$ generated by forms of the
same degree $\delta$.  Let $\mathscr{R}(I)=\oplus_{j=0}^{\infty}I^j$ be the Rees algebra of
$I$. The special fiber ring $k\otimes_R \mathscr{R} (I)$ describes the homogeneous coordinate
ring of the image of the rational map induced by $I$. As a special case this construction yields homogeneous coordinate rings of Gauss images and of secant varieties.
It is important to compute the multiplicity of the special fiber ring.
For this goal, observe
$k\otimes_R \mathscr{R} (I)\simeq k[I_{\delta}]\subseteq k[R_{\delta}]=R^{(\delta)}$, where
$k[I_{\delta}]$ is the $k$-algebra generated by the forms in $I$
of degree $\delta$ and $R^{(\delta)}$ is the $\delta$-th Veronese subring of $R$.
After
rescaling the grading of both $k[I_{\delta}]$ and $R^{(\delta)}$,
$k[I_{\delta}]\subseteq R^{(\delta)}$ is a homogeneous inclusion
of standard graded $k$-algebras. If we could express
$e(k[I_{\delta}])$ in terms of $e(R^{(\delta)})$ and local
multiplicities along ${\rm Proj}(R^{(\delta)})$, then we can
compute  $e(k[I_{\delta}])$ using the multiplicities of $R$. Indeed, since the $i$-th homogeneous component of $R^{(\delta)}$ is the $\delta i$-th homogeneous component of $R$,
 $e(R^{(\delta)})=e(R)\cdot \delta^{d-1}$. Furthermore as ${\rm Proj}(R^{(\delta)})\simeq {\rm
Proj}(R)$,  the local multiplicities along ${\rm
Proj}(R^{(\delta)})$ do not change if we pass to the local
multiplicities along ${\rm Proj}(R)$.

For $A\subseteq B$ as above, it is easy to see that ${\rm
dim}\,A\leq {\rm  dim}\,B$. If ${\rm dim}\,A={\rm dim}\,B$, the
first case to consider is when ${\rm dim}\,B/A_1B=0$, i.e.,
$B$ is integral over $A$.  If we assume further that $A_p$ is
reduced and ${\rm rank}_{A_p} B_p=r$ for every prime ideal $p$ in
$A$ of  dimension $d$, then $e(B)=re(A)$. The questions are as follows: what happens when ${\rm
dim}\,B/A_1B\geq 1$ and how to find those local multiplicities along ${\rm Proj}(B)$?
In 2001, Simis, Ulrich and Vasconcelos gave
some partial answers:

\begin{theo}{\rm(\hspace{-.004in}\cite[6.4]{SUV})} Let
$A\subseteq B$ be a homogeneous inclusion of standard graded
Noetherian rings of the same dimension, with $A_0=B_0$  each Artinian
local and $B$
 equidimensional. Let $\mathfrak{P}=\{p\in
{\rm Spec}(A)\,|\, {\rm dim}\, A/p={\rm dim}\, A\}$ and assume
that $A_p$ is reduced and ${\rm rank}_{A_p} B_p=r$ for every prime ideal
$p\in \mathfrak{P}$. If ${\rm dim}\,B/A_1B=1$, then
\begin{align}
e(B)= re(A)+\sum_{q \in \mathfrak{Q}} e_{A_1B_q}(B_q)\, e(B/q),
\end{align}
where $\mathfrak{Q}=\{q\in V(A_1B)\cap {\rm Proj}(B) \,|\,{\rm dim}\,B/q={\rm
dim}\,B/A_1B\}$.
\end{theo}

Here $e(A)$ and $e(B)$ are the usual multiplicities of $A$ and
$B$. We use $e_{A_1B_q}(B_q)$ to denote the Hilbert-Samuel
multiplicity of $B_q$ with respect to the ideal $A_1B_q$.

The assumption ${\rm dim}\,B/A_1B=1$ forces  the prime ideals in the
projective spectrum of $B$  containing the ideal $A_1B$ to be
minimal over $A_1B$. Therefore   the ideal $A_1B_q$ is primary to
the maximal ideal $qB_q$, for $q\in \mathfrak{Q}$, and one can use the Hilbert-Samuel
multiplicity $e_{A_1B_q}(B_q)$. But when ${\rm dim}\,B/A_1B>1$,
the sum will involve prime ideals which are not minimal over
$A_1B$ and the Hilbert-Samuel multiplicity is not defined locally
at those prime ideals.
Thus for ${\rm dim}\,B/A_1B>1$, these authors  showed only that
  the left hand side of
  Equation (1) is greater than the
right hand side \cite{SUV}. In this theorem,
 the required local multiplicities along ${\rm Proj}(B)$ are just the Hilbert-Samuel multiplicities.

Later in 2007,  Validashti  improved this inequality when  ${\rm
dim}\,B/A_1B>1$:

\begin{prop} {\rm(\hspace{-.004in}\cite[5.9.4]{V})}
Let
$A\subseteq B$ be a homogeneous inclusion of standard graded
Noetherian rings of the same dimension, with $A_0=B_0$ each Artinian
local and $B$  equidimensional.
Let $\mathfrak{P}=\{p\in {\rm Spec}(A)\,|\,
{\rm dim}\, A/p={\rm dim}\, A\}$ and assume that $A_p$ is reduced
and ${\rm rank}_{A_p} B_p=r$ for every prime ideal $p\in \mathfrak{P}$. Then
\begin{align}
e(B)\geq re(A)+\sum_{q \in \mathfrak{Q}^{\prime}} j
(A_1B_q)\,e(B/q),
\end{align}
where $\mathfrak{Q}^{\prime}=\{q\in V(A_1B)\cap {\rm Proj}(B)
\,|\,\ell (A_1B_q)={\rm
 dim}\,B_q\}$.
 \end{prop}

Here $\ell(A_1B_q)$ is the analytic spread of the ideal $A_1B_q$ and $j(A_1B_q)$ is its $j$-multiplicity in the sense of Achilles-Manaresi \cite{AM}\cite{AM1}.
Observe the set $\mathfrak{Q}^{\prime}$ is finite (it is contained
 in the set of the centers of
the Rees valuations of $A_1B$). But
 it  may contain prime ideals $q$ which are not
minimal over $A_1B$.   Hence the $j$-multiplicity is used in
Inequality (2) to replace the Hilbert-Samuel multiplicity. Since
$A_1B_q$ has maximal analytic spread, one has $j(A_1B_q)\neq 0$.
 Also if $q$ is minimal over $A_1B$,
 $j(A_1B_q)=e_{A_1B_q}(B_q)$.
 Thus when  ${\rm dim}\,B/A_1B>1$, some terms are
added to the right hand side of Equation (1) to make it closer to
$e(B)$. But Validashti \cite{V} also
gave an example to show that  Inequality (2) can be strict. In
Theorem 4.1, we give the extra terms on the right hand side of
Inequality (2) required to yield an equality for arbitrary
dimensions of $B/A_1B$. This solves the problem when
${\rm dim}\,A={\rm dim}\, B$. In Theorem 4.9, we also provide a formula for the case
when ${\rm dim}\,A<{\rm dim}\,B$. Thus we give a complete answer to the original question.

As  mentioned before, these formulas can be used to find the
multiplicity of the special fiber ring $k\otimes_R \mathscr{R}
(I)$, where $I\subseteq R$ is an ideal generated by forms of the
same degree in a standard graded $k$-algebra $R$.  This yields an
upper bound for the reduction number of $I$ with respect to any
reduction. It also provides a formula for the degree of the image
of the rational map induced by  $I$. In particular, it gives the
degree of dual varieties of hypersurfaces. The first formula
relating the degree of the dual variety to the degree of the
variety itself was given by Pl\"{u}cker  in 1834 for complex plane
curves
with at most nodes and cusps as singularities \cite[p.857]{Kline}.
Later, in 1975 Teissier generalized Pl\"{u}cker's formula to
hypersurfaces with at most isolated singularities
\cite[App.II]{T}. In 1994, Kleiman
 generalized Teissier's Pl\"{u}cker formula to  projective
 varieties with at most isolated singularities \cite[Theorem 2]{K}.
  In 1997, based on Kleiman's work,  Thorup generalized the Pl\"{u}cker formula to projective
 varieties with arbitrary singularities  using the Chow groups of
 the varieties
 \cite{TH}. Both of their formulas assume the variety has
 non-deficient dual (i.e., the dual variety is a hypersurface).
 In that case the degree of the dual variety is called the {\it
 class} of the variety.
   Our formulas can be
used to find the degree of the dual variety for any hypersurface
without any restriction on its singularities and dual variety (i.e., we do not need the dual variety to be a hypersurface).
In particular, it gives a generalization of Teissier's Pl\"{u}cker
formula to hypersurfaces with arbitrary singularities. See Section 5 for these applications.

This paper is divided into five parts. In Section 2,  first we fix notation which will be used
throughout the paper. After that we define the $j$-multiplicity for ideals
that have possibly non-maximal analytic spread and prove some
facts about the $j$-multiplicity and (super-) reduction
sequences. In Section 3, we express the multiplicity of a standard
graded Noetherian algebra
by means of local $j$-multiplicities
of various hyperplane sections. This formula becomes simpler  if the ideal $I$ satisfies  condition $G_{t+1}$ and Artin-Nagata property $AN_{t-1}^-$, or if $I$ is   a
complete intersection for every $q\in V(I)$ with ${\rm ht}\,q\leq t$,  or is perfect  of
height two satisfying condition $G_{t+1}$, or is Gorenstein  of height three satisfying condition $G_{t+1}$ (see Corollaries 3.3 and 3.5).  In Section 4, we
consider   a homogeneous inclusion $A\subseteq B$ of standard
graded Noetherian rings over an Artinian local ring and express
the multiplicity of $A$ in terms of that of $B$ and of local
$j$-multiplicities of hyperplane sections along ${\rm Proj}(B)$;
this is done by applying the formulas obtained in Section 3. Finally
in Section 5, we obtain the formulas for the multiplicity of
special fiber rings and give some applications.

\section{ Preliminaries.}
\noindent

In this section, first we  fix notation and recall some basic
concepts and results which will be used throughout the paper. Then
we define the $j$-multiplicity for ideals that have possibly
non-maximal analytic spread and prove some facts about the
$j$-multiplicity and (super-) reduction sequences.

Throughout the paper, let $(R, m, k)$ be either a Noetherian local
ring or a standard graded Noetherian algebra over an Artinian
local ring $(R_0,m_0)$, where  $m$ is either the maximal ideal or
the  homogeneous maximal ideal $(m_0,R_1R) $ of $R$ and $k=R/m$ is
the residue field. Let $I$  be an ideal of $R$. Write
$G=$ ${\rm gr}_{I}(R)=\oplus_{j=0}^{\infty} I^j/I^{j+1}$ for the
{\it associated graded ring} of $R$ with respect to the ideal $I$
and $F=G/mG$ for the {\it special fiber ring} of this ideal. The
{\it analytic spread} of $I$ is defined by $\ell(I)={\rm dim}\,F$
and one has ${\rm ht}\,I\leq \ell(I)\leq {\rm dim}\,R={\rm
dim}\,G$, where ${\rm ht}\,I$ denotes the height of $I$.

If $I$ is an $m$-primary ideal of $R$,  we write $e_I(R)$  for the
{\it Hilbert-Samuel multiplicity} of $R$ with respect to $I$ and
$e(R)$ for $e_m(R)$. If $I$ is not necessarily $m$-primary but has
maximal analytic spread, i.e., $\ell (I)={\rm dim}\, R$,  one
writes $j(I)$ for the $j$-{\it multiplicity} of $R$ with respect
to $I$ (see \cite{AM}, and Definition 2.1 and subsequent remarks below). Notice that if $I$ is $m$-primary then $j(I)=e_I(R)$ \cite{AM}.

For an ideal $I$ and a submodule $N$ of an $R$-module $M$,
$N:_M\langle I\rangle=\cup_{i\geq 0}(N:_MI^i)=\{x\in
M\,|\,I^ix\subseteq N \,{\rm for\, some} \,\,i\in \mathbb{N}\}$.
An element $a\in I$ is said to be {\it a filter-regular element }
with respect to $I$ if  $0:_R a\subseteq 0:_R \langle I \rangle$.
This is equivalent to saying that $a$ is not in any $q\in {\rm Ass}(R)$
such that $I\nsubseteq q$. A sequence of elements $a_1, \ldots,
a_t$ of $I$ is called {\it a filter-regular sequence} for $R$ with
respect to $I$ if $(a_1, \ldots, a_{i-1})R:_Ra_i\subseteq (a_1,
\ldots, a_{i-1})R:_R\langle I\rangle$ for $1\leq i \leq t$. Assume that
$\ell(I)=s$. Then a sequence of elements $a_1,\ldots,a_s$ of $I$
is called  {\it a reduction sequence} for $I$ if the initial forms
$a_1^*, \ldots, a_s^*$ of $a_1, \ldots, a_s$ in $G$
 are of degree one and form a {\it filter-regular sequence} for
$G$ with respect to $G_+$, where  $G_+$ is the ideal generated by
all homogeneous  elements of positive degree in $G$, and a system
of parameters for $F$. If $s={\rm dim }\,R=d$, a reduction
sequence $a_1, \ldots, a_d$ for $I$ will be called a {\it
super-reduction} for $I$, if for every relevant
highest-dimensional prime ideal $P$ of $G$,  the initial forms $a_1^*, \ldots, a_{d(P)}^*$ are a
system of parameters for $G/(mG+P)$, where $d(P)={\rm dim
}\,G/(mG+P)$. Suppose $a_1, \ldots, a_d$
form a super-reduction for $I$ and $J_{d-1}=(a_1, \ldots,
a_{d-1})R:_R \langle I\rangle$; then $j(I)=\lambda_R
(R/(J_{d-1}+a_dR))$, where $\lambda_R (R/(J_{d-1}+a_dR))$ denotes
the length of $R/(J_{d-1}+a_dR)$ \cite{AM}. It is well-known that if $R$ has
infinite residue field, every ideal has reduction sequences or, if
the ideal has maximal analytic spread, super-reductions. For
concepts and results about analytic spread, ($j$-)multiplicities and
(super-) reduction sequences, see \cite{AM}, \cite{AM1},
\cite{AM2}, \cite{FM}, 
\cite{SA},
\cite{SE} and \cite{V}.

Now we  will define the $j$-multiplicity for ideals that have
possibly non-maximal analytic spread.

\begin{defi}
\em Let $(R,m,k)$ be either a Noetherian local
ring or a standard graded Noetherian algebra over an Artinian
local ring $(R_0,m_0)$, where  $m$ is either the maximal ideal or
the  homogeneous maximal ideal $(m_0,R_1R) $ of $R$ and $k=R/m$ is
the residue field. Let $I$ be an
$R$-ideal with analytic spread $\ell (I)=s$. Set $G={\rm gr}_I(R)$
and $\Gamma_m(G)=0:_G \langle m\rangle=\oplus_{j=0}^{\infty}
\Gamma_m\,(I^j/I^{j+1})$. The module $\Gamma_m(G)$ is  finitely
generated and  graded over a standard graded Noetherian algebra
over the Artinian local ring $R/m^{t}$ for some integer $t>0$.
Furthermore,  ${\rm dim}\,\Gamma_m(G)\leq s$. So $\Gamma_m(G)$ has
a Hilbert function that is eventually a polynomial of degree at
most $s-1$. Define
$$
j_s(I)=(s-1)! {\it \mathop {\lim }\limits_{i \to \infty
}}\frac{\lambda(\Gamma_m(I^i/I^{i+1}))}{i^{s-1}}.
$$
\end{defi}

\bigskip

Notice that $j_s(I)=e(\Gamma_m(G))$ if
${\rm dim }\,\Gamma_m(G)= s$ and zero otherwise. When $s={\rm
dim}\,R$, $j_s(I)$ is the usual $j$-multiplicity $j(I)$. In particular if $I$ is $m$-primary, $j_s(I)=j(I)=e_I(R)$, the Hilbert-Samuel multiplicity.

The following lemma shows that the $j_s$-multiplicity does not change modulo the ideal $0:_R \langle I\rangle$.

\begin{lemm}
 Let $(R,m,k)$  be as in Definition 2.1 with
 $k$ infinite and  $I$  an ideal of $R$. Set $J=0:_R
\langle I\rangle$ and assume
$\bar{R}=R/J\neq 0$. Write $\bar{I}=I\bar{R}$ and
$s=\ell(\bar{I})$. Then $s\geq 1$ and the following hold:
\begin{enumerate}
 \item [$($a$)$] $\ell(I)=\ell(\bar{I})=s$.
\item [$($b$)$] $j_{s}(I)=j_{s}(\bar{I})$.
\end{enumerate}
\end{lemm}

\begin{proof}
(a)  If $s=0$, then $I$ is in the nilradical of $R$ and $\bar{R}=R/J= 0$.
Hence it is easy to see that $s\geq 1$ and $\ell(I)\geq s$. On the
other hand, since $k$ is infinite, there exists an ideal $H\subseteq I$ generated by $s$
elements with $I^{j}\subseteq HI^{j-1}+J$ for $j>>0$. So
$I^{j}\subseteq HI^{j-1}+J \cap I^{j} $ for sufficiently large $j$. Let $t$ be
 an integer such that $J=0:_R I^t$. By the Artin-Rees Lemma,
there exists $c\geq 0$ so that  for all $j\geq c+t$, $I^{j}\cap
J=I^{j-c}(I^{c}\cap J)\subseteq I^{t}J=0$. So $I^{j}\subseteq
HI^{j-1}$ for $j>>0$, i.e., $H$ is a reduction of $I$. Since $H$
is generated by $s$ elements, it follows that $\ell(I)\leq s$.

(b) Let $G=\mbox{gr}_I(R)=\oplus_{j=0}^{\infty}I^j/I^{j+1}$,
$\bar{G}=\mbox{gr}_{\bar{I}}(\bar{R})=\oplus_{j=0}^{\infty}I^j/(I^j\cap
J+I^{j+1})$.
 By the proof of part(a), $I^j\cap J=0$ when $j>>0$. Set $\bar{m}=m\bar{R}$, one has
$\Gamma_m(I^j/I^{j+1})=\Gamma_{\bar{m}}( I^j/(I^j\cap J+I^{j+1}))$
for $j>>0$. Thus one has ${\rm dim}\,\Gamma_m (G)={\rm
dim}\,\Gamma_{\bar{m}} (\bar{G})$. Furthermore
$j_{s}(I)=e(\Gamma_m(G))=e(\Gamma_{\bar{m}}(\bar{G}))=j_s(\bar{I})$
if dim $\Gamma_m(G)=s$, and $j_s(I)=0=j_s(\bar{I})$ otherwise.
\end{proof}

\bigskip

Next we want to prove a result about (super-)reduction sequences.
Before doing that, let us recall the bigraded ring $T={\rm
gr}_m({\rm gr}_I(R))$ from \cite{AM1}. Observe
$T=\oplus_{i,j=0}^{\infty} T_{ij}$, where
$T_{ij}=(m^iI^j+I^{j+1})/(m^{i+1}I^j+I^{j+1})$ and $T_{00}=R/m=k$.
First we give a fact about the bigraded ring $T$.

\begin{prop}
 Let $(R,m,k)$ be as in Definition 2.1 and assume that $k$ is infinite.
   Let $I=(a_1,\ldots,a_n)R$ be an ideal
 with
$\ell(I)=s$. Write $x_i=\sum_{j=1}^n \lambda_{ij}a_j$ for
$1\leq i\leq s$ and $\Lambda=(\lambda_{ij})\in R^{sn}$. Then there
is a dense open subset $U$ of $k^{sn}$ such that if the image
$\overline{\Lambda}=(\overline{\lambda_{ij}})\in U$,  the
images $x_1^o,\ldots,x_s^o$ of $x_1,\ldots,x_s$ in $T_{01}=I/mI$
are a filter-regular sequence for $T$ with respect to the ideal
$T_{01}T$ and a system of parameters for
$F=G/mG=\oplus_{j=0}^{\infty}T_{0j}$.
\end{prop}

\begin{proof}
Observe that if $s=0$, the result is obvious. So we may assume $s>0$. First we show there exists a dense open subset $U_1$ of $k^{sn}$
such that if the image $\overline{\Lambda}\in U_1$, then
$x_1^o,\ldots,x_s^o$  form a filter-regular sequence  with respect
to the ideal $T_{01}T$.
 To do this, let $Z=(z_{ij})$ be variables over $R$,
$1\leq i \leq s$, $1\leq j \leq n$, $R^{\prime}=R[Z]$ and
$T^{\prime}={\rm gr}_{mR^{\prime}}({\rm
gr}_{IR^{\prime}}(R^{\prime}))=T[Z]$. Write the generic linear combinations $
x_i^{\prime}=\sum_{j=1}^n z_{ij} a_j,\, 1\leq i\leq s $.  The images $x_1^{\prime
o},\ldots,x_s^{\prime o}$  in
$T^{\prime}_{01}=IR^{\prime}/mIR^{\prime}$ are of degree one (set
the degrees of $z_{ij}$ to be zero) and form a filter-regular
sequence with respect to $T^{\prime}_{01}T^{\prime}$, i.e.,  a
weakly regular sequence locally at every prime ideal of
Spec$(T^{\prime})\setminus V(T^{\prime}_{01}T^{\prime})$. So the
Koszul complex $K.(x_1^{\prime o},\ldots,x_s^{\prime o})$ in
$T^{\prime}$ is acyclic locally on ${\rm Spec}(T^{\prime})\setminus
V(T^{\prime}_{01}T^{\prime})$. Notice
$Z-\overline{\Lambda}=(z_{ij}-\overline{\lambda_{ij}})$ is an ideal in $k[Z]$ and
$K.(x_1^{\prime o},\ldots,x_s^{\prime
o})\otimes_{k[Z]}k[Z]/(Z-\overline{\Lambda})=K.(x_1^{ o},\ldots,x_s^{ o})$ is
a Koszul complex in  $T$. We only need to show that, by avoiding
 a proper closed subset of $k^{sn}$, $K.(x_1^{ o},\ldots,x_s^{ o})_Q$ is acyclic whenever $Q\in {\rm Spec}(T)\backslash V(T_{01}T)$. Observe that $T^{\prime}$ is a
finitely generated algebra over the Noetherian domain $k[Z]$ and
the $0$-th Koszul homology $H_0(x_1^{\prime o},\ldots,x_s^{\prime o})
=T^{\prime}/(x_1^{\prime o},\ldots,x_s^{\prime
o})T^{\prime}$ is
a  finite $T^{\prime}$-module. By the Generic Flatness Lemma,
there exists an element $0\neq f\in k[Z]$ such that
$T_f^{\prime}$ and $H_0(x_1^{\prime o},\ldots,x_s^{\prime o})_f$ are all free over
$k[Z]_f$.  Set $U_1=k^{s n}\backslash V(f)$.  We claim that $U_1$ is the desired dense open
subset.
Indeed, let $\Lambda\in R^{sn}$ with $\overline{\Lambda}\in
U_1$; then $K.(x_1^{\prime o},\ldots,x_s^{\prime
o})_f\otimes_{k[Z]_f}k[Z]_f/(Z-\overline{\Lambda})\simeq K.(x_1^{ o},\ldots,x_s^{ o})$.
For every prime ideal $Q\in {\rm Spec}(T)\backslash V(T_{01}T)$,  $T^{\prime}_Q/(x_1^{\prime o},\ldots,x_s^{\prime
o})T^{\prime}_Q=(T^{\prime}_f/(x_1^{\prime o},\ldots,x_s^{\prime
o})T^{\prime}_f)_{\bar{Q}}$  is flat over $k[Z]_f$. Thus ${\rm
Tor}_i^{k[Z]_f}(T^{\prime}_Q/(x_1^{\prime o},\ldots,x_s^{\prime
o})T^{\prime}_Q, k[Z]_f/(Z-\overline{\Lambda}))=0$ for $1\leq i\leq s$, i.e.,
$K.(x_1^{ o},\ldots,x_s^{ o})_Q$ is acyclic.

Secondly, by \cite{NR} there exists another dense open subset $U_2$
of $k^{sn}$ such that if the image $\overline{\Lambda}\in
U_2$, then
$x_1,\ldots,x_s$ generate a minimal reduction of $I$. Therefore
$x_1^o,\ldots,x_s^o$ form a system of parameters for the special
fiber ring $F=G/mG$. Now let $U=U_1\cap U_2$.
 Then $U$ is a dense open subset of $k^{sn}$ and has the required
 property.
\end{proof}

\begin{rema}
\em In the paper, by abuse of notation we will call $\Lambda=(\lambda_{ij})\in R^{sn}$
as in Proposition 2.3  general elements in $R^{sn}$.
\end{rema}

As a result of Proposition 2.3, we have the following corollary.

\begin{coro}
 Let $(R,m,k)$ be as in Definition 2.1 and assume that $k$ is  infinite.
   Let $I=(a_1,\ldots,a_n)R$ be an ideal
 with
$\ell(I)=s$. Then for general elements $\Lambda=(\lambda_{ij})\in
R^{sn}$, $x_1,\ldots,x_s$ form a reduction sequence for $I$ or, if
$s={\rm dim} \,R$, a super-reduction for $I$.
\end{coro}

\begin{proof}
 Notice that if
the images $x_1^o,\ldots,x_s^o$ in $T_{01}=I/mI$ are a
filter-regular sequence with respect to the ideal $T_{01}T$,
then the initial forms $x_1^*,\ldots,x_s^*$ in the associated graded ring
$G={\rm gr}_I(R)$ form a filter-regular sequence of degree one
with respect to $G_+$ \cite{AM1}. Now the first part follows from Proposition 2.3.  For the second part, when $s={\rm
dim}\,R=d$, we can just avoid finitely many more proper closed
subsets of $k^{dn}$ to assume that $x_1,\ldots,x_d$ form a
super-reduction for the ideal $I$.
\end{proof}

\begin{rema}
\em Corollary 2.5 was first stated by Achilles and Manaresi
\cite[2.9]{AM}. But their proof  shows only that one can choose
general elements $(\lambda_{ij})$ sequentially. Our proof shows
that indeed one can choose the general elements all at once.
\end{rema}

\section{ Formulas for a graded algebra over an Artinian local ring.}
\noindent

In this section, we will give the multiplicity formula
for a standard graded Noetherian algebra $R$ over an Artinian local
ring. This formula expresses the multiplicity of $R$ in terms of data associated to prime ideals in $V(I)$,
for $I\subseteq R$ any given ideal generated by linear forms.

\begin{theo}
Let $R=R_0[R_{1}]$ be a standard graded Noetherian ring of dimension $d$ with
$(R_0, m_0)$ an Artinian local ring. Assume
$|R_0/m_0|=\infty$. Let $I=(a_1,\ldots,a_n)R$ be an ideal
generated by homogeneous elements $a_1,\ldots,a_n$ of degree one.
Write ${\rm ht}\,I=g$ and $\ell(I)=s$. For general elements
$\Lambda=(\lambda_{ij})\in R^{s n}$, let $ x_i=\sum_{j=1}^n
\lambda_{ij}a_j, $  $ J_{i-1}=(x_1,\ldots,x_{i-1})R :_R \langle
I\rangle$,
$\mathfrak{Q}_0=\{q\in {\rm Min}(I)\,|\,{\rm
dim}\,R/q=d\}$
and  $\mathfrak{Q}_i=\{q\in {\rm
Min}(J_{i-1}+I)\,|\,{\rm dim}\,R/q=d-i\}$, for $1\leq i \leq s$.
Then $\ell (I_q/(x_1,\ldots,x_{i-1})R_{q})=1$ for every  $q \in
\mathfrak{Q}_i$ with $1\leq i\leq s$, and
\begin{align}
 e(R)=\sum_{q\in \mathfrak{Q}_0}e_{I_q}(R_q)\,e(R/q)+
 \sum_{i={\rm max}\{1,\,g \}}^{s}\sum_{ q\in \mathfrak{Q}_i
 }j_{1}(\frac{I_{q}}{(x_1,\ldots,x_{i-1})R_{q}})\,e(R/q).
\end{align}
\end{theo}

\begin{proof}  By Corollary 2.5,
there exist general elements $\Lambda=(\lambda_{ij}) \in R^{sn}$
such that $x_1,\ldots,x_s$ form a reduction sequence for $I$. In particular
by \cite[2.8]{AM}, $x_1,\ldots,x_s$ are a filter-regular sequence
for $R$ with respect to the ideal $I$.  We show that whenever
$x_1,\ldots,x_s$ form a reduction sequence for $I$, then
\begin{align}
e(R)=\sum_{q\in\mathfrak{Q}_0}\lambda(R_q)\,e(R/q)+\sum_{i=1}^{s}\sum_{q\in
\mathfrak{Q}_i}\lambda(\frac{R_{q}}{J_{i-1}R_q+x_{i}R_{q}})\,e(R/q).\notag
\end{align}
We use induction on $d$ to prove this. First when $d=0$, one has $g=0$ and
$\mathfrak{Q}_0=\{m\}$, where $m=(m_0, R_1R)$ is the homogeneous maximal ideal of $R$.  Then $e(R)=\lambda(R)=\lambda (R_m)$ and
we are done. Let $d\geq 1$ and
$\mathfrak{Q}_0^{\prime}=\{q\in {\rm Min}(R)\,|\,{\rm dim}\,R/q=d,
I\nsubseteq q\}$.  By the associativity formula,
\begin{align}
e(R)
=\sum_{q\in \mathfrak{Q}_0}\lambda(R_{q})\,e(R/q)+ \sum_{q\in
\mathfrak{Q}_0^{\prime}}\lambda(R_{q})\,e(R/q).\notag
\end{align}
We may assume $s>0$ as otherwise
$\mathfrak{Q}_0^{\prime}=\emptyset$ and the result is obvious. Furthermore  if ${\rm
dim}\,R/J_0<d$, then $\mathfrak{Q}_0^{\prime}=\emptyset$ and $\mathfrak{Q}_i=\emptyset$ for $1\leq i\leq s$. The last assertion follows since
$x_i$ is a non zerodivisor on $R/J_{i-1}$ for $1\leq i\leq s$.  Assume ${\rm dim}\,R/J_0=d$. Since $x_1$ is a linear non
zerodivisor on $R/J_0$,
$$
\sum_{q\in
\mathfrak{Q}_0^{\prime}}\lambda(R_{q})\,e(R/q)=e(R/J_0)=e(R/(J_0+x_1R)).
$$
 Let $\mathfrak{Q}_1^{\prime}=\{q\in {\rm Min}(J_0+x_1R)\,|\,{\rm
dim}\,R/q=d-1, I\nsubseteq q\}$, then
\begin{align}
e(R)
&=\sum_{q\in \mathfrak{Q}_0}\lambda(R_{q})\,e(R/q)+e(R/(J_0+x_1R))\notag\\
&=\sum_{q\in \mathfrak{Q}_0}\lambda(R_{q})\,e(R/q)+\sum_{q\in
\mathfrak{Q}_1}\lambda(\frac{R_{q}}{J_0R_q+x_1 R_q})\,e(R/q)
+\sum_{q\in
\mathfrak{Q}_1^{\prime}}\lambda(\frac{R_{q}}{J_0R_q+x_1
R_q})\,e(R/q).
\end{align}
 Similarly, we may suppose $s>1$
as otherwise $\mathfrak{Q}_1^{\prime}=\emptyset$ and the result is
obvious. Set $\overline{R}=R/(J_0+x_1R)$; then
${\rm dim} \,\overline{R}=d-1$. Let $G_I(J_0+x_1R,R)$ be the
initial ideal of $J_0+x_1R$ in the associated graded ring $G={\rm
gr}_I(R)$. Then $\overline{G}={\rm
gr}_{I\overline{R}}(\overline{R})=G/G_I(J_0+x_1R,R)$ and
$\overline{F}=\overline{G}/\overline{m}\overline{G}=G/(G_I(J_0+x_1R,R)+mG)$.
We will show $\ell(I\overline{R})={\rm dim}\,\overline{F}= s-1$.
Observe $x_1^*,\ldots,x_s^*$ form a system of parameters for
$F=G/mG$. Thus ${\rm dim} \,G/(x_1^*G+mG)=s-1$. Since
$x_1R\subseteq J_0+x_1R\subseteq J_1$,
$G_I(J_0+x_1R,R):_G \langle G_+\rangle= G_I(J_1,R):_G \langle
G_+\rangle=x_1^*G:_G \langle G_+\rangle$ (see \cite[3.2]{AM1}).
Therefore $G_I(J_0+x_1R,R)$ and
$x_1^*G$ have the same relevant associated prime ideals. As
${\rm dim}\,G/(x_1^*G+mG)=s-1>0$, it follows that   ${\rm dim}\,\overline{F}={\rm dim}\,G/(G_I(J_0+x_1R,R)+mG)={\rm dim}\,G/(x_1^*G+mG)=s-1$.

Let $\overline{x_2},\ldots,\overline{x_s}$ be the images of $x_2,\ldots,x_s$ in $\overline{R}$. From the above argument, $G_I(J_0+x_1R,R):_G \langle G_+\rangle=x_1^*G:_G \langle G_+\rangle$.
Hence the initial forms $\overline{x_2}^*,\ldots,\overline{x_s}^*$ of $\overline{x_2},\ldots,\overline{x_s}$ in $\overline{G}$ form a filter-regular sequence for $\overline{G}$ with respect to $\overline{G}_+$. As $x_1^*$ is part of a
system of parameters of $G/mG$,\,
  $\overline{x_2},\ldots,\overline{x_s}$  form a
reduction sequence for $\overline{I}$. Observe
$(J_0,x_1,\ldots,x_{i-1})R:_R \langle I\rangle=J_{i-1}$ for $1\leq i\leq s$. By
the induction hypothesis on $\overline{R}$, we get
$$
e(R/(J_0+x_1R))=\sum_{i=1}^{s}\sum_{q\in
\mathfrak{Q}_i}\lambda(\frac{R_{q}}{J_{i-1}R_q+x_{i}R_{q}})\,e(R/q).
$$
Substituting into  Equation (4) we are done.

Now let $q\in \cup_{i=0}^s \mathfrak{Q}_i$.
If $q\in \mathfrak{Q}_0$, then $\lambda(R_q)=e_{I_q}(R_q)$.
So assume
 $q\in \mathfrak{Q}_i$ for some $i\geq 1$. We will show
$\ell (I_q/(x_1,\ldots,x_{i-1})R_{q})=1$ and
\begin{align}
\lambda(R_{q}/(J_{i-1}R_q+x_{i}R_{q}))=
j_{1}(I_{q}/(x_1,\ldots,x_{i-1})R_{q}).
 \end{align}
First we want to see
$$
\lambda(R_q/(J_{i-1}R_q+x_{i}R_q))
=j_{1}(I(R_q/J_{i-1}R_q)).\notag
$$
Let $\overline{R}=R/J_{i-1}$. Since $q\in V(J_{i-1}+I)$ with ${\rm
dim}\, R/q=d-i$, ${\rm ht}\,J_{i-1}\geq i-1$ and $x_i\in I$ is a non zerodivisor on $R/J_{i-1}$,
one has dim $\overline{R}=d-i+1$.
In the local ring $\overline{R}_{\overline{q}}$, $1\geq {\rm
dim}\,\overline{R}_{\overline{q}}\geq {\rm
grade}\,I\overline{R}_{\overline{q}}\geq 1 $. Therefore
$I\overline{R}_{\overline{q}}$ is
$q\overline{R}_{\overline{q}}$-primary with
$\ell(I\overline{R}_{\overline{q}})={\rm dim}\,\overline{R}_{\overline{q}}=1$. By \cite[3.8 and 2.6]{AM},  we only need
to show the initial form $\overline{x_{i}}^*$  of
$\overline{x_{i}}$ in the associated graded ring ${\rm
gr}_{I\overline{R}_{\overline{q}}}(\overline{R}_{\overline{q}})$
is filter-regular
 with respect
to $({\rm
gr}_{I\overline{R}_{\overline{q}}}(\overline{R}_{\overline{q}}))_+$.
But this comes from  \cite[3.2]{AM1} and the fact that the
filter-regular property  is preserved under localization.

Next we need to show $ \ell (I_q/(x_1,\ldots,x_{i-1})R_{q})=1$,
 and
$$
j_{1}(I\overline{R}_{\overline{q}})=
j_{1}(\frac{I_{q}}{(x_1,\cdots,x_{i-1})R_{q}}).\notag
$$
For this we only need to apply Lemma 2.2 to the local ring
$R_{q}/(x_1,\cdots,x_{i-1})R_{q}$ and its ideal
$I_{q}/(x_1,\cdots,x_{i-1})R_{q}$.

Finally since $\mathfrak{Q}_i=\emptyset$, $0\leq i\leq g-1$, we
get Equation (3).
\end{proof}

\begin{rema}
\em Theorem 3.1 is motivated by the intersection algorithms
constructed by  Achilles and Manaresi \cite{AM1}. Indeed the idea behind the original St\"{u}ckrad-Vogel algorithm of refined intersection theory (cf. \cite[Section 3.2]{AM2}, for example), from which \cite{AM1} evolved, is that the $j$-multiplicity can be calculated by adding together the contributions from top-dimensional components, then cutting by a hyperplane in general position and repeating the process, and so on. Returning to \cite{AM1}, recall that the
degrees of cycles $\upsilon_i(\underline{x},R)$ of $R$ supported
on $V(I)$ are defined by
\begin{align}
{\rm
deg}\,&(\upsilon_0(\underline{x},R))=\sum_{q\in
\mathfrak{Q}_0}\lambda
(R_{q})\,e(R/q),\notag\\
{\rm
deg}\,(\upsilon_i(\underline{x},R))&=\sum_{q\in
\mathfrak{Q}_i}\lambda
\big(\frac{R_{q}}{J_{i-1}R_q+x_{i}R_{q}}\big)\,e(R/q),\,\,1\leq i
\leq s.\notag
\end{align}
They  defined the multiplicity sequence $c_0(I),\ldots,c_d(I)$ with
respect to the ideal $I$ using
 the bigraded ring $T={\rm gr}_m({\rm gr}_I(R))$.
 They  proved that if the images $x_1^o,\ldots,x_s^o$ in $T_{01}=I/mI$
are a filter-regular sequence for $T$ with respect to the ideal
$T_{01}T$ and a system of parameters for $F=G/mG$, then
 ${\rm
deg}(\upsilon_{i}(\underline{x},R))=c_{d-i} (I) $ \cite[4.1]{AM1}.
 The proof of Theorem
3.1 shows that for general elements $(\lambda_{ij})$ in $R^{sn}$,
$c_d(I)=\sum_{q\in \mathfrak{Q}_0}e_{I_q}(R_q)\,e(R/q)$, $c_{d-i}
(I)=\sum_{ q\in \mathfrak{Q}_i
 }j_{1}(\frac{I_{q}}{(x_1,\ldots,x_{i-1})R_{q}})\,e(R/q)$ for $1\leq i\leq s$, and the others are all zero.
\end{rema}

Now we want to  apply Theorem 3.1 to some classes of ideals to obtain better formulas. To do
this, first we need to recall some facts about residual
intersections from \cite{HU}.  Let $R$ be a
Noetherian ring and $I$ an ideal of $R$. Set
$H=(x_1,\ldots,x_t)R:_RI$, where $(x_1,\ldots,x_t)R\subsetneq I$.
If ${\rm ht}\,H\geq t\geq$ ${\rm ht}\,I$, then $H$ is said to be a
{\it $t$-residual intersection} of $I$ with respect to
$(x_1,\ldots,x_t)R$. Furthermore if $I_q=(x_1,\ldots,x_t)R_q$ for
all $q\in V(I)$ with ${\rm ht}\,q\leq t$, then  $H$ is a {\it
geometric $t$-residual intersection} of $I$. Notice that the ideal $H$ is a
geometric $t$-residual intersection of $I$ if and only if ${\rm
ht}\,H\geq t$ and ${\rm ht}(H+I)\geq t+1$.
An ideal $I$ satisfies
condition $G_{t+1} $ if $\mu (I_p)\leq$ ${\rm ht}\,p$ for all $p\in V(I)$
 such that ${\rm ht}\,p\leq t$. Here $\mu(M)$
denotes the least number of generators of a module $M$.

Recall  that an ideal $I$ has the {\it Artin-Nagata property} $AN_t^-$
 if for every $i$ with ${\rm ht}\,I \leq i \leq t$
and every geometric $i$-residual intersection $H$ of $I$, $R/H$ is Cohen-Macaulay \cite{U}.
 The
ideal $I$ is said to be {\it strongly Cohen-Macaulay} (SCM) if the
Koszul homology modules of any set of generators of $I$ are
Cohen-Macaulay modules. Notice that it suffices to check this property
for a fixed system of generators of $I$ \cite{HU}.

The following corollary shows that condition $G_{t+1}$ and Artin-Nagata property $AN_{t-1}^-$ determine
how many terms we can simplify in Equation (3).

\begin{coro}
Let $R=R_0[R_{1}]$ be a standard graded Cohen-Macaulay  ring of
 dimension $d$ with $(R_0,m_0)$ an Artinian local ring.  Assume
 $|R_0/m_0|=\infty$. Let $I=(a_1,\ldots,a_n)R$ be an ideal
generated by homogeneous elements $a_1,\ldots,a_n$  of degree one. Write ${\rm ht}\, I=g$, $\ell(I)=s$
 and assume that the ideal $I$ satisfies  condition $G_{t+1}$ and Artin-Nagata property $AN_{t-1}^-$, where $g\leq t \leq s$.
 For
general elements $\Lambda=(\lambda_{ij})\in R^{s n}$, define $
x_i$  and $\mathfrak{Q}_i$, $1 \leq i \leq s$, as in Theorem
3.1.  Let $ H_{i}=(x_1,\ldots,x_{i})R :_R
I$, where $0 \leq i \leq s$.  Then
\begin{align}
e(R)= e(R/I)+
 \sum_{i=g+1}^{t}e(R/(H_{i-1}+I))+\sum_{i=t+1}^{s}\sum_{ q\in \mathfrak{Q}_i
 }j_{1}(\frac{I_{q}}{(x_1,\ldots,x_{i-1})R_{q}})\,e(R/q).
\end{align}
\end{coro}

\begin{proof}
Let $Z=(z_{ij})$, $1\leq i\leq s, 1\leq j\leq n$, be variables
over $R$, and $R^{\prime}=R[Z]$. Set $x_i^{\prime}=\sum_{j=1}^n
 z_{ij}a_j$  and
 $H_{i}^{\prime}=(x_1^{\prime},\ldots,x_{i}^{\prime})R^{\prime}:_{R^{\prime}}
 IR^{\prime}$, $0\leq i\leq s$. Since $I$ satisfies condition $G_{t+1}$, by \cite[3.2]{HU} for each $i$ with $g\leq i\leq t$,
 $H_i^{\prime}$ is a geometric  $i$-residual intersection of
 $IR^{\prime}$, i.e., ${\rm ht}\, H_i^{\prime}\geq i$ and ${\rm ht}(H_i^{\prime}+IR^{\prime})\geq
 i+1$.
 Let $k=R/(m_0,R_1R)=R_0/m_0$, $\Lambda=(\lambda_{ij})\in
 R^{sn}$ and $\overline{\Lambda}$ be the image of  $\Lambda$ in $k^{sn}$.
 Write
 $\pi(H_i^{\prime})$ for the ideal in $R$ generated by the image of $H_i^{\prime}$
 under the evaluation map sending $z_{ij}$ to $\lambda_{ij}$.
 By \cite[3.1]{HHU}, for
 all $i$ with $g\leq i\leq t$, there exists a dense open subset $U_1$ of $k^{sn}$
 such that ${\rm ht}(\pi(H_i^{\prime}))
 \geq i$ and ${\rm ht}(\pi(H_i^{\prime})+I)\geq i+1$ whenever $\overline{\Lambda}\in U_1$.
 Let $U_2$ be the dense open subset of $k^{sn}$ as in Theorem 3.1;  $U=U_1\cap
 U_2$ is still a dense open subset of $k^{sn}$. Let $\Lambda\in R^{sn}$
 with
 $\overline{\Lambda}\in U$.
 Then for $g\leq i\leq t$, since  $
 \pi(H_i^{\prime})\subseteq H_i$, $H_i$ is also a geometric  $i$-residual intersection of
 $I$ and hence $I_q=(x_1,\ldots,x_i)R_q$
  for every $q\in
\mathfrak{Q}_i$.
Now consider $\mathfrak{Q}_i$, where $g\leq i\leq t$.
If $q\in \mathfrak{Q}_g$, since
$I_q=(x_1,\ldots,x_g)R_q$ is a complete intersection,
$e_{I_q}(R_q)=\lambda(R_q/I_q)$ if $g=0$, and
$j_1(I_q/(x_1,\ldots,x_{g-1})R_q)=e_{I_q}(R_q)=\lambda(R_q/I_q)$ if $g>0$. Fix $i$ with $g+1\leq i \leq t$ and let $J_{i-1}$ be defined as in Theorem 3.1. Since $I$ satisfies  $AN_{t-1}^-$, $R/H_{i-1}$ is Cohen-Macaulay and therefore $J_{i-1}=H_{i-1}$. Also from \cite[1.7]{U}, ${\rm ht}\,(J_{i-1}+I)=i$ and hence $\mathfrak{Q}_i\neq \emptyset$.
Let $q\in \mathfrak{Q}_i$. By
Equation (5),
$j_1(I_q/(x_1,\ldots,x_{i-1})R_q)=\lambda(R_q/(J_{i-1}R_q+x_iR_q))=\lambda(R_q/(H_{i-1}+I)_q)$.
Applying  Theorem 3.1 and the
associativity formula,
\begin{align}
e(R)
 =&\sum_{q\in \mathfrak{Q}_g}\lambda(R_q/I_q)\,e(R/q)\notag\\
 +&\sum_{i=g+1}^{t}\sum_{ q\in \mathfrak{Q}_i
 }\lambda(R_q/(H_{i-1}+I)_q)\,e(R/q)+\sum_{i=t+1}^{s}\sum_{ q\in \mathfrak{Q}_i
 }j_{1}(\frac{I_{q}}{(x_1,\ldots,x_{i-1})R_{q}})\,e(R/q)\notag\\
 =& \,e(R/I)+
 \sum_{i=g+1}^{t}e(R/(H_{i-1}+I))+\sum_{i=t+1}^{s}\sum_{ q\in \mathfrak{Q}_i
 }j_{1}(\frac{I_{q}}{(x_1,\ldots,x_{i-1})R_{q}})\,e(R/q).\notag
\end{align}
\end{proof}

\bigskip

\begin{rema}
\em In Corollary 3.3,  we can replace the Artin-Nagata property $AN_{t-1}^-$ by
strongly Cohen-Macaulay. Indeed if $I$ is a strongly Cohen-Macaulay ideal which satisfies condition $G_{t+1}$, then by \cite[3.3]{HU} (see also \cite[3.1]{H}), for $g \leq i \leq
t$, $H_{i}^{\prime}$ is a geometric $i$-residual intersection
of $IR^{\prime}$, $R^{\prime}/H_{i}^{\prime}$ is Cohen-Macaulay and ${\rm
ht}\,H_{i}^{\prime}=i$.  By the proof of Corollary 3.3, for general elements $\Lambda=(\lambda_{ij})\in R^{sn}$,   $H_i$ is  a geometric  $i$-residual intersection of
 $I$ for  $g \leq i \leq
t$. Notice
$Z-\Lambda=(z_{ij}-\lambda_{ij})$ is a regular sequence on both
$R^{\prime}$ and $R^{\prime}/IR^{\prime}$. For $g+1 \leq i \leq
t$, write
$\pi(H_{i-1}^{\prime})$ for the image modulo
$(Z-\Lambda)R^{\prime}$. By \cite[4.7]{HU} (see also \cite[10.4]{KU}), we have $\pi(H_{i-1}^{\prime})=H_{i-1}$
and $R/H_{i-1}$ is also Cohen-Macaulay. Hence
$J_{i-1}=H_{i-1}$.
 Furthermore by the proof of  \cite[1.7]{U}, ${\rm ht}\,(J_{i-1}+I)=i$ and therefore $\mathfrak{Q}_i\neq \emptyset$. The remaining proof is the same as that of  Corollary 3.3.
\end{rema}

When the ideal $I$
 is a complete intersection for every  $q\in
V(I)$ with ${\rm ht}\,q\leq t$,   or is perfect of height 2 satisfying  condition $G_{t+1}$, or is Gorenstein of
height 3 satisfying condition $G_{t+1}$, we obtain nicer formulas.

\begin{coro}
 Let $R=R_0[R_{1}]$ be a  standard
graded Cohen-Macaulay   ring of
 dimension $d$ with $(R_0,m_0)$ an Artinian local ring.  Assume
 $|R_0/m_0|=\infty$. Let $I=(a_1,\ldots,a_n)R$ be an ideal generated by homogeneous elements
$a_1,\ldots,a_n$  of degree one. Write ${\rm
ht }\,I=g$ and $\ell(I)=s$. For general elements
$\Lambda=(\lambda_{ij})\in R^{s n}$,  let $\Lambda_{i-1}$ be the sub-matrix
consisting of the first $i-1$ rows of $\Lambda$ and define $x_i$,
$\mathfrak{Q}_i$ as before, for $1\leq i
\leq s$.

$($a$)$ Assume $I_q$ is a complete intersection for every $q\in
V(I)$ with ${\rm ht}\,q\leq t$, where $g\leq t\leq s$. Then
\begin{align}
e(R)& = e(R/I)+
 \sum_{i=g+1}^{t}e(R/(F_{i-1}+I))
  +\sum_{i=t+1}^{s}\sum_{ q\in \mathfrak{Q}_i
 }j_{1}(\frac{I_{q}}{(x_1,\ldots,x_{i-1})R_{q}})\,e(R/q),\notag
\end{align}
where $F_{i-1}={\rm
 Fitt}_0(I/(x_1,\ldots,x_{i-1})R)$.

$($b$)$ Assume $I$ is a perfect ideal of height $2$ which
satisfies condition $G_{t+1}$.  Write $\mu(I)=n$ and let
$X_{n\times (n-1)}=(x_{ij})$ be a matrix such that $I=I_{n-1}(X)$.
Then
\begin{align}
e(R)=e(R/I)+
 \sum_{i=3}^{t}e(R/(I_n(X \,|\, \Lambda_{
 i-1}^T)+I))+\sum_{i=t+1}^{s}\sum_{ q\in \mathfrak{Q}_i
 }j_{1}(\frac{I_{q}}{(x_1,\ldots,x_{i-1})R_{q}})\,e(R/q).\notag
\end{align}

$($c$)$ Assume $I$ is a perfect Gorenstein ideal of height $3$
which satisfies condition $G_{t+1}$. Write $\mu(I)=n$ and let
$X_{n\times n}=(x_{ij})$ be an alternating matrix such that
$I={\rm Pf}_{n-1}(X)$ $($the ideal generated by the $n-1$ by $n-1$
Pfaffians of $X$$)$. Let
 $T_{i-1}=\begin{pmatrix} X & \Lambda_{i-1}^T\\
 -\Lambda_{i-1} & 0 \end{pmatrix}$,
 define
 $J(T_{i-1})$ to be the $R$-ideal generated by the Pfaffians of all
 principal sub-matrices of  $T_{i-1}$ which contain $X$ for $ 4\leq
 i\leq
 t $ $($see {\rm \cite{KU}}$)$.
 Then
\begin{align}
e(R)=e(R/I)+
 \sum_{i=4}^{t}e(R/(J(T_{i-1})+I))+\sum_{i=t+1}^{s}\sum_{ q\in \mathfrak{Q}_i
 }j_{1}(\frac{I_{q}}{(x_1,\ldots,x_{i-1})R_{q}})\,e(R/q).\notag
\end{align}
\end{coro}

\begin{proof}
In both three cases,  $I$ satisfies condition
 $G_{t+1}$. Write $k=R/(m_0, R_1R)=R_0/m_0$. By the proof of Corollary 3.3, for all $i$ with $g\leq i\leq t$,
  there is a dense open subset $U\subseteq k^{sn}$ such that whenever $\overline{\Lambda}
  =(\overline{\lambda_{ij}})\in U$,
 $H_i=(x_1,\cdots,x_{i})R:_{R}I$ is  a geometric
$i$-residual intersection of
 $I$, and Equation (3) holds. Set $\Lambda=(\lambda_{ij})\in R^{sn}$
 with $\overline{\Lambda}\in U$.

(a) As in the proof of Corollary 3.3, let $Z=(z_{ij})$, $1\leq
i\leq s, 1\leq j\leq n$, be variables over $R$
 and $R^{\prime}=R[Z]$. Let $x_i^{\prime}=\sum_{j=1}^n
 z_{ij}a_j$,
 $H_{i}^{\prime}=(x_1^{\prime},\ldots,x_{i}^{\prime})R^{\prime}:_{R^{\prime}}
 IR^{\prime}$ and $J_{i}=(x_1,\ldots,x_{i})R:_R \langle I\rangle$ for $0\leq i\leq s$.
Consider the local ring $R_q$, where $q\in
V(I)$ with ${\rm ht}\,q= t$.
 Let  $g^{\prime}={\rm ht}\,I_q$.  We may
assume $I_q=(a_1,\ldots,a_{g^{\prime}})_q$ after elementary transformations. There is an invertible $n$ by
$n$ matrix $E$ over $R_q$ such that
$(a_1,\ldots,a_n)^T=E(a_1,\ldots,a_{g^{\prime}},0,\ldots,0)^T$. Fix $i$ with $g^{\prime}+1\leq i\leq t$. Observe that $R_q^{\prime}=R_q[Z]$ and
$(H_{i-1}^{\prime})_q=(x_1^{\prime},\cdots,x_{i-1}^{\prime})R^{\prime}_q:_{R^{\prime}_q}IR^{\prime}_q$.
Let $Z_{i-1}$ be the sub-matrix of  $Z=(z_{ij})$
generated by the first $i-1$ rows. Then $\varphi=Z_{i-1} E$ is a
generic $i-1$ by $n$ matrix over $R_q.$
 If we write
$\varphi=(U \,|\, W) $, where $U$ is a $i-1$ by $g^{\prime}$ sub-matrix,
then
$(x_1^{\prime},\ldots,x_{i-1}^{\prime})^T=U(a_1,\ldots,a_{g^{\prime}})^T$.
By \cite[3.4]{HU},
$(H_{i-1}^{\prime})_q=(x_1^{\prime},\ldots,x_{i-1}^{\prime})_q+I_{g^{\prime}}
(U)_q.$

In the  Cohen-Macaulay local ring $R^{\prime}_q$, since $IR^{\prime}_q$
is a complete intersection, the ideal $IR^{\prime}_q$ is strongly
Cohen-Macaulay and satisfies condition $G_{t+1}$. By the same argument of Remark 3.4, we have
$R_q/(H_{i-1})_q$ is Cohen-Macaulay and
$(J_{i-1})_q=(H_{i-1})_q=\pi((H_{i-1}^{\prime})_q)=(x_1,\ldots,x_{i-1})_q+I_{g^{\prime}}
(\pi(U))_q$.

Observe there is an exact sequence
$$
R_q^{\begin{pmatrix}{g^{\prime}}\\
2 \end{pmatrix} }\oplus R_q^{i-1} \stackrel{(d_1 |\pi(U)^T
)}{\longrightarrow} R_{q}^{g^{\prime}}
{\longrightarrow}I_q/(x_1,\ldots,x_{i-1})_q\rightarrow 0,
$$
where $d_1$ is the $1$-th derivation in the Koszul complex $K(a_1, \ldots, a_{g^{\prime}})$:
$$\cdots R_q^{\begin{pmatrix}{g^{\prime}}\\
2 \end{pmatrix} } \stackrel{d_1
}{\longrightarrow} R_{q}^{g^{\prime}}\stackrel{d_0
}
{\longrightarrow}R_q\rightarrow 0.  $$
So we have
$$
{\rm Fitt}_0(I/(x_1,\ldots,x_{i-1}))_q+I_q=(H_{i-1})_q+I_q=I_{g^{\prime}}
(\pi(U))_q+I_q.
$$
Thus $(J_{i-1}+I)_q=(H_{i-1}+I)_q={\rm
Fitt}_0(I/(x_1,\ldots,x_{i-1}))_q+I_q$.

Finally let  $q\in V(I)$ with ${\rm ht}\,q=t$ and ${\rm ht}\,I_q=g$. For every $i$ with $g+1\leq i\leq t$, $R_q/(H_{i-1})_q$ is Cohen-Macaulay and hence ${\rm ht}\,(J_{i-1}+I)_q=i$ (see the proof of \cite[1.7]{U}).
Therefore  $\mathfrak{Q}_i\neq \emptyset$ for every $i$ with $g+1\leq i\leq t$. We are done by the proof of
Corollary 3.3 and the associativity formula.

(b) Let $R^{\prime}=R[W, Z]$, where $W=(w_{ij})$ is a generic
$n\times (n-1)$ matrix, $Z=(z_{ij})$ is a generic $s\times n$
matrix  and $I^{\prime}=I_{n-1}(W)=(\Delta_1,\ldots,\Delta_n)$, in an obvious notation.
Set
$(x_1^{\prime},\ldots,x_s^{\prime})^T=Z(\Delta_1,\ldots,\Delta_n)^T$
and
$H_{i}^{\prime}=(x_1^{\prime},\ldots,x_{i}^{\prime})R^{\prime}:_{R^{\prime}}I^{\prime}R^{\prime}$
for $0\leq i\leq s$. Fix $3\leq i\leq t$. By \cite[4.1]{H}, we
have $H_{i-1}^{\prime}=I_n(W \,|\, Z_{i-1}^T)$. Let
$(W-X,Z-\Lambda)=(w_{ij}-x_{ij};z_{ij}-\lambda_{ij})$ and write
$\pi(H_{i-1}^{\prime})$ for the image modulo
$(W-X,Z-\Lambda)R^{\prime}$. By a similar argument as in the
proof of \cite[10.5]{KU}, we have  $R/H_{i-1}$ is Cohen-Macaulay and
$H_{i-1}=\pi(H_{i-1}^{\prime})=I_n(X \,|\, \Lambda_{i-1}^T)$.
Hence $J_{i-1}=H_{i-1}=I_n(X \,|\, \Lambda_{i-1}^T)$ and we are
done by the proof of Corollary 3.3.

(c)
 By the proof of
 \cite[10.5]{KU}, for each $i$ with $4\leq i\leq t$, we have
 $R/H_{i-1}$ is Cohen-Macaulay and $H_{i-1}=J(T_{i-1})$. Hence $J_{i-1}=H_{i-1}=J(T_{i-1})$ and
  we are done by the proof of Corollary 3.3.
\end{proof}

\bigskip

\section{Homogeneous inclusions of graded algebras over an Artinian local ring.}

\noindent

In this section, we will consider homogeneous inclusions of two
standard graded Noetherian algebras over an Artinian local ring. First we
consider the case where these two rings have the same dimension.

\begin{theo}
 Let $A=A_0[A_1]\subseteq B=B_0[B_{1}]$ be a homogeneous
inclusion of standard graded Noetherian rings of the same
dimension $d$. Let $A_0=B_0=R_0$ be an Artinian local ring with
maximal ideal $m_0$ and $|R_0/m_0|=\infty$. Write
$A_1A=(a_1,\ldots,a_n)A$, where $a_1,\ldots,a_n$ are homogeneous
elements of degree one, and  ${\rm ht}\, A_1B=g$. Let
$\mathfrak{P}=\{p\in {\rm Spec}(A)\,|\, {\rm dim}\, A/p={\rm dim}\,A\}$ and
assume that $A_p$ is reduced and ${\rm rank}_{A_p} B_p=r$ for
every prime ideal $p\in \mathfrak{P}$. For general elements $\Lambda =
(\lambda_{ij})\in R_0^{d n}$, define $ x_i, \,J_{i},
\,\mathfrak{Q}_i$, $0 \leq i \leq d$, as  before. Then
\begin{align} e(B)=re(A)+
 \sum_{q\in \mathfrak{Q}_0}e_{A_1B_q}(B_q)\,e(B/q)+
 \sum_{i={\rm max}\{1,\,g \}}^{d-1}\sum_{ q\in \mathfrak{Q}_i
 }j_{1}(\frac{A_1B_{q}}{(x_1,\ldots,x_{i-1})B_{q}})\,e(B/q).\notag
\end{align}
\end{theo}

\begin{proof}
Recall $G={\rm gr}_{A_1B}(B)$ and $T={\rm gr}_m({\rm
gr}_{A_1B}(B))$, where $m=(m_0, B_1B)$ is the homogeneous maximal
ideal of $B$. Then $F=G/mG\simeq A/m_0A$ (see \cite{SUV}). Since ${\rm dim}\, A/m_0A={\rm dim}\,
A=d$,  $\ell(A_1B)=d$.  By Proposition 2.3,   for
general elements $\Lambda = (\lambda_{ij})\in R_0^{d n}$, the images $x_1^o,\ldots,x_d^o$ of $x_1,\ldots,x_d$ in
$T_{01}=I/mI$ are a filter-regular sequence with respect to the
ideal $T_{01}T$ and a system of parameters for $F$. By the proof
of Theorem 3.1, Equation (3) holds for such $\Lambda$. Thus we
only need to show:
$$
re(A)=j_{1}(A_1B_{m}/(x_1,\ldots,x_{d-1})B_{m}).
$$
By Remark 3.2 and \cite[Theorem 4.1]{AM1},
\begin{align}
j_{1}(A_1B_{m}/(x_1,\ldots,x_{d-1})B_{m})=
c_0(A_1B)=j_d(A_1B).\notag
\end{align}
 By \cite[6.1]{SUV}, we know $e(G)=e(B)=re(A)+e_{\infty}(A,B)$, where $e_{\infty}(A,B)=e(G/0:_G\langle B_1G\rangle)$. Since $B_1B$
and $m$ have the same radical, $0:_G\langle
B_1G\rangle=0:_G\langle mG\rangle$. Therefore
$j_d(A_1B)=e(0:_G\langle mG \rangle)=e(G)-e_{\infty}(A, B)=re(A)$
(see also \cite{V}).
\end{proof}

\bigskip

Theorem 4.1 is a generalization of  Theorem 1.1. When $B$ is
equidimensional, we have the following corollary.

\begin{coro}
In the same setting as Theorem 4.1, assume $B$ is equidimensional
and let $\mathfrak{Q}^{\prime}=\{q\in V(A_1B)\cap {\rm Proj}(B)
\,|\,\ell (A_1B_q)={\rm
 dim}\,B_q\}$ as in Proposition 1.2. For general elements
$\Lambda = (\lambda_{ij})\in R_0^{d n}$, define $ x_i, J_{i}$ and
$\mathfrak{Q}_i$, $0 \leq i \leq d$, as before. Then
\begin{align} e(B)=re(A)+
 \sum_{q\in \mathfrak{Q}^{\prime}}j(A_1B_q)\,e(B/q)+
 \sum_{i={\rm max}\{1,\,g \}}^{d-1}\sum_{ q\in \mathfrak{Q}_i\backslash \mathfrak{Q}^{\prime}
 }j_{1}(\frac{A_1B_{q}}{(x_1,\ldots,x_{i-1})B_{q}})\,e(B/q).\notag
\end{align}
\end{coro}

\begin{proof}
Let  $q\in \mathfrak{Q}^{\prime}$. First we show that
$q$ is contained in  $\mathfrak{Q}_i$ for some $i$ with $g\leq i\leq d-1$.
Indeed for such $q$, write $\ell(A_1B_q)={\rm dim}\,B_q=i$. Then
$J_{i-1}+A_1B\subseteq q$, since otherwise $(J_{i-1})_q=B_q$ and
$\sqrt{A_1B_q}=\sqrt{(x_1,\ldots,x_{i-1})B_q}$, which contradicts
 the equality $\ell(A_1B_q)=i$. Also ${\rm dim}\,B/q=d-i$ because
$B$ is equidimensional. Thus $q\in \mathfrak{Q}_i$. Notice that
$\mathfrak{Q}^{\prime}$ is a fixed finite set which does not
depend on $x_1,\ldots,x_d$. Moreover $R_0/m_0$ is a subfield of
$k(q)=B_q/qB_q$ for every prime ideal $q$ of $B$. By avoiding
finitely many more proper closed subsets of $k^{dn}$,
 we can choose general elements
$\Lambda=(\lambda_{ij})\in R_0^{dn}$ such that
$x_1,\ldots,x_i$ form a super-reduction for the ideal $A_1B_q$
for each $q\in \mathfrak{Q}^{\prime}$ with ${\rm ht}\,q=i$. Thus
$j_{1}(\frac{A_1B_{q}}{(x_1,\ldots,x_{i-1})B_{q}})=
\lambda(B_q/(J_{i-1}B_q+x_iB_q))=j(A_1B_q)$ \cite{AM}. We are done
by applying Theorem 4.1.
\end{proof}

\bigskip
From Corollary 4.2, we can see that we have indeed added extra
terms to the right hand side of Inequality (2) in Proposition 1.2,
to obtain equality.

When  $B$ is a domain with the same dimension as that of $A$, we have the
following corollary.

\begin{coro}
 Let $A=k[A_1]\subseteq B=k[B_{1}]$ be a homogeneous
inclusion of standard graded Noetherian domains of the same
dimension $d$, where $k$ is an infinite field.  Write
$A_1A=(a_1,\ldots,a_n)A$, where $a_1,\ldots,a_n$ are homogeneous
elements of degree one,  and ${\rm ht}\, A_1B=g$. Let $r=[L: K]$,
where $L={\rm Quot}(B)$  and $K={\rm Quot}(A)$. For general
elements $\Lambda = (\lambda_{ij})\in k^{d n}$, define $ x_i$,
$J_{i-1}$, and $\mathfrak{Q}_i$, $1 \leq i \leq d$, as before. Then
\begin{align} e(B)=re(A)+
\sum_{i=g}^{d-1}\sum_{ q\in \mathfrak{Q}_i
 }j_{1}(\frac{A_1B_{q}}{(x_1,\ldots,x_{i-1})B_{q}})\,e(B/q).\notag
\end{align}
\end{coro}

\bigskip

By Theorem 4.1, and Corollaries 3.3 and 3.5, it is easy to get the
following corollary.

\begin{coro}
In the same setting as Theorem 4.1, assume
$B$ is  Cohen-Macaulay.   For general elements
$\Lambda=(\lambda_{ij})\in R_0^{d n}$, define  $ x_i$, $ H_{i-1}$ and  $\Lambda_{i-1}$, $1\leq i \leq d$, as before.

$($a$)$ Assume the ideal $A_1B$ satisfies condition $G_d$ and Artin-Nagata property $AN^-_{d-2}$.  Then
\begin{align}
e(B)=re(A)+ e(B/A_1B)+
 \sum_{i=g+1}^{d-1}e(B/(H_{i-1}+A_1B)).\notag
\end{align}

$($b$)$ Assume $A_1B_q$ is a complete intersection for every $q\in
{\rm Proj}\, (B)$.  Then
\begin{align}
e(B)=re(A)+e(B/A_1B)+
  \sum_{i=g+1}^{d-1}e(B/(F_{i-1}+A_1B)),\notag
\end{align}
where $F_{i-1}={\rm
 Fitt}_0(A_1B/(x_1,\ldots,x_{i-1})B)$.

$($c$)$ Assume  $A_1B$ is a perfect ideal of height $2$ which
satisfies condition $G_d$. Write $\mu(A_1B)=n$ and let $X_{n\times
(n-1)}$ be a matrix such that $A_1B=I_{n-1}(X)$.   Then
\begin{align}
e(B)=re(A)+e(B/A_1B)+
  \sum_{i=3}^{d-1}e(B/(I_n(X \,|\, \Lambda_{i-1
 }^T)+A_1B)).\notag
\end{align}

$($d$)$
  Assume $A_1B$ is a perfect Gorenstein ideal of height $3$
which satisfies condition $G_{d}$. Write $\mu(A_1B)=n$ and let
$X_{n\times n}$ be an alternating matrix such that
$A_1B={\rm Pf}_{n-1}(X)$. Let $T_{i-1}=\begin{pmatrix} X & \Lambda_{i-1}^T\\
 -\Lambda_{i-1} & 0 \end{pmatrix}$ and
 $J(T_{i-1})$ be the $B$-ideal generated by the Pfaffians of all
 principal sub-matrices of  $T_{i-1}$ which contain $X$ for $4\leq i\leq d-1$. Then
\begin{align}
e(B)=re(A)+e(B/A_1B)+
 \sum_{i=4}^{d-1}e(B/(J(T_{i-1})+A_1B)).\notag
\end{align}
\end{coro}

\begin{rema}
\em In part(a) of Corollary 4.4, we can replace the Artin-Nagata property  $AN^-_{d-2}$ by  strongly
Cohen-Macaulay (see the argument of Remark 3.4).
\end{rema}

Now we consider the case when ${\rm dim}\, A\neq {\rm dim}\,B$. Let $k$ be an
infinite field and $R$ a finitely generated $k$-algebra. Recall that
${\rm deg}_k(R)={\rm min}\{{\rm rank}_S R\}$, where $S$ ranges over all Noether normalizations of $R$.
If $R=k[a_1,\ldots,a_n]$, where $a_1,\ldots,a_n$
are homogeneous elements of degree one,   ${\rm deg}_k(R)=e(R)={\rm rank}_S R$, for any Noether normalization $S$ of $R$ which is a $k$-algebra
generated by linear combinations of $a_1,\ldots,a_n$ \cite{SUV}.

Write ${\rm dim}\,R=d$. Let $Y=(y_{ij})$, $1\leq i\leq d,
1\leq j \leq n$, be variables over $R$. Consider the generic
linear combinations $y_i=\sum_{j=1}^n y_{ij}a_j$, $1\leq i \leq
d$, regarded as elements of $R^{\prime}=R\otimes_k
k^{\prime}=k^{\prime}[a_1,\ldots,a_n]$, where
$k^{\prime}=k(y_{ij})$. By \cite[7.3]{FUV},
$S^{\prime}=k^{\prime}[y_1,\ldots,y_d]\subseteq R^{\prime}$ is a
finite homogeneous inclusion. Hence ${\rm deg}_{k^{\prime}}(R^{\prime})={\rm
rank}_{S^{\prime}}R^{\prime}=e(R^{\prime})=e(R)={\rm deg}_k(R)$.

For $A=A_0[A_1]\subseteq B=B_0[B_{1}]$, a homogeneous
inclusion of standard graded Noetherian rings with $A_0=B_0=R_0$   Artinian local, the
associated graded ring $G={\rm
gr}_{A_1B}(B)=R_0[A_1T][\overline{B_1}]$, where $T$ is a variable
over $B$ and $\overline{B_1}$ is the image of $B_1$ in $G$. By
assigning bi-degree $(1,0)$ to the elements of $\overline{B_1}$
and bi-degree $(0,1)$ to the elements of $A_1T$,
$G=\oplus_{i,j=0}^{\infty}G_{ij}$ is a bi-graded ring with
$\oplus_{j=0}^{\infty}G_{0j}=R_0[A_1T]\simeq A$ (see \cite{SUV}).  Thus we can think that $A\subseteq G$ and  $A_p\otimes_A G$ is a
standard graded Noetherian $A_p$-algebra for every prime ideal $p$ of $A$.

\begin{prop}
Let $A=A_0[A_1]\subseteq B=B_0[B_{1}]$ be a homogeneous
inclusion of standard graded Noetherian rings with ${\rm dim}\,A=s$ and ${\rm dim}\,B=d$.
 Let
 $A_0=B_0=R_0$ be an Artinian local ring with
 maximal ideal $m_0$ and $|R_0/m_0|=\infty$. Write
$A_1A=(a_1,\ldots,a_n)A$, where $a_1,\ldots,a_n$ are homogeneous
elements of degree one.  Let
$\mathfrak{P}=\{p\in {\rm Spec}(A)\,|\, {\rm dim}\, A/p={\rm dim}\,A\}$ and
assume that $A_p$ is reduced,  ${\rm dim}\,A_p\otimes_A G=d-s$ and ${\rm deg}_{A_p}(A_p\otimes_A G)=r$ for
every prime ideal $p\in \mathfrak{P}$, where $G={\rm gr}_{A_1B}(B)$. For general elements $\Lambda =
(\lambda_{ij})\in R_0^{s n}$,
 let $ x_i=\sum_{j=1}^n \lambda_{ij}a_j, $ $1\leq i\leq s$,
 $ J_{s-1}=(x_1,\ldots,x_{s-1}) B:_B \langle A_1B\rangle$ and
$\mathfrak{Q}_s$ be the set of all prime ideals $q$ in ${\rm Min}(
J_{s-1}+A_1B)$ with   ${\rm dim}\,B/q=d-s$. Then
\begin{align}
\sum_{q\in
\mathfrak{Q}_s}j_{1}(\frac{A_1B_{q}}{(x_1,\ldots,x_{s-1})B_{q}})\,e(B/q)=r e(A).\notag
\end{align}
\end{prop}

\begin{proof}
Observe $\ell(A_1B)={\rm dim}\,A=s$ (see the proof of Theorem 4.1). By Proposition 2.3
and \cite{AM1} (see Remark 3.2), for general elements
$\Lambda=(\lambda_{ij})\in R_0^{sn}$,
$$\sum_{q\in
\mathfrak{Q}_s}j_{1}(\frac{A_1B_{q}}{(x_1,\ldots,x_{s-1})B_{q}})\,e(B/q)
={\rm deg}\,\upsilon_s(\underline{x},B)
={\rm
deg\,}\upsilon_s(\underline{x}^*,G)=c_{d-s}(A_1B),$$
 which   does not depend on the
choice of $\Lambda$. We only need to show $c_{d-s}(A_1B)=re(A)$. Observe $G=R_0[A_1T][\overline{B_1}]$ and
$(A_1T)G=G_+=\oplus_{i=0}^{\infty}\oplus_{j=1}^{\infty}G_{ij}$.
When $s=d$, by the proof of Theorem 4.1, $c_0(A_1B)={\rm
deg\,}\upsilon_d(\underline{x}^*,G)=j((A_1T)G)=$ $({\rm rank}_A\,G)\cdot
e(A)=$ $r e(A)$. Suppose $s<d$ and let $l=d-s$.
Set $A=R_0[A_1]=R_0[a_1,\ldots,a_n]\subseteq
B=R_0[B_1]=R_0[b_1,\ldots,b_{\tau}]$. Then $A\simeq
R_0[a_1T,\ldots,a_nT]\subseteq G=R_0[a_1T,\ldots,a_nT,
\overline{b_1},\ldots,\overline{b_{\tau}}]$, where
$\overline{b_1},\ldots,\overline{b_{\tau}}$ are the images of $b_1,
\ldots,b_{\tau}$ in $G_{10}$. We may assume $A=
R_0[a_1T,\ldots,a_nT]= R_0[G_{01}]$. Let $\widetilde{R_0}=R_0[Z, Y, W]$ and $R_0^{\prime}=(\widetilde{R_0})_{m_0\widetilde{R_0}}$, where
$Z=(z_{ij})$, $1\leq i\leq s$, $1\leq j\leq n$,
$Y=(y_{\mu \nu})$, $1\leq \mu\leq l$, $1\leq
\nu\leq \tau$, $W=(w_{\alpha \beta})$, $1\leq \alpha \leq l$, $1\leq
\beta \leq n+l$, are variables over $G$. Let
$G^{\prime}=G\otimes_{R_0} R_0^{\prime}=R_0^{\prime}[G_{10},G_{01}]$,
$A^{\prime}=A\otimes_{R_0} R_0^{\prime} =R_0^{\prime}[G_{01}]$ and
$A^{\prime\prime}=R_0^{\prime}[G_{01},
y_1^{\prime},\ldots,y_l^{\prime}]$, where
$y_{\mu}^{\prime}=\sum_{\nu=1}^{\tau}y_{\mu \nu}\overline{b_{\nu}}$
for $1\leq \mu \leq l$.  For every $p^{\prime}\in {\rm Spec}(A^{\prime})$ with ${\rm dim}\,A^{\prime}/p^{\prime}={\rm dim}\,A^{\prime}=s$, $p^{\prime}=pA^{\prime}$ for some $p\in \mathfrak{P}$. Hence  $A^{\prime}_{p^{\prime}}$ is a field, $A^{\prime}_{p^{\prime}}\otimes_{A^{\prime}}\,G^{\prime}$ is a finitely generated standard graded
$A^{\prime}_{p^{\prime}}$-algebra with dimension equal to $l$ and ${\rm deg}_{A^{\prime}_{p^{\prime}}}(A^{\prime}_{p^{\prime}}\otimes_{A^{\prime}}\,G^{\prime})
={\rm deg}_{A_p}(A_p\otimes_A G)=r$ (see the argument before Proposition 4.6).
By \cite[7.3]{FUV}, ${\rm dim}\,A^{\prime}_{p^{\prime}}\otimes_{A^{\prime}}A^{\prime\prime}={\rm dim}\,A^{\prime}_{p^{\prime}}\otimes_{A^{\prime}}G^{\prime}=l=d-s$.
Hence
${\rm dim}\,A^{\prime\prime}={\rm dim}\,G^{\prime}=d$.
 Notice for every minimal prime ideal $p^{\prime\prime}$ of $A^{\prime\prime}$ with ${\rm dim}\,A^{\prime\prime}/p^{\prime\prime}=d$, $p^{\prime\prime}$ is the extension $p^{\prime}A^{\prime\prime}$, where $p^{\prime}$ is a prime ideal of $A^{\prime}$ with ${\rm dim}\,A^{\prime}/p^{\prime}=s$. Thus $A^{\prime\prime}_{p^{\prime\prime}}$ is reduced and ${\rm rank}_{A^{\prime\prime}_{p^{\prime\prime}}}\,G^{\prime}_{p^{\prime\prime}}={\rm rank}_{A^{\prime}_{p^{\prime}}\otimes_{A^{\prime}}A^{\prime\prime}}\,A^{\prime}_{p^{\prime}}\otimes_{A^{\prime}}G^{\prime}
={\rm deg}_{A^{\prime}_{p^{\prime}}}(A^{\prime}_{p^{\prime}}\otimes_{A^{\prime}} G^{\prime})=r$ (see also the argument before Proposition 4.6).

Now
consider the ideal
$I=(G_{01}, y_1^{\prime},\ldots,y_l^{\prime})G^{\prime}$. Let
$y_{\alpha}^{\prime\prime}=\sum_{\beta=1}^{n}w_{\alpha \beta}a_{\beta}T+\sum_{\beta=n+1}^{n+l}w_{\alpha \beta}y_{\beta-n}^{\prime}$, $1\leq \alpha \leq l$,
$x_i^{\prime}=\sum_{j=1}^{n}z_{ij}a_jT$, $1\leq i\leq s$; then
$y_1^{\prime\prime},\ldots,y_l^{\prime\prime},x_1^{\prime},\ldots,x_s^{\prime}$
form a super-reduction for $I$. Indeed, it is easy to see
that $y_1^{\prime\prime},\ldots,y_l^{\prime\prime}$ are filter-regular with respect to $I$, as they are generic linear combinations of the generators of $I$. Moreover, for every prime ideal $P^{\prime}\in {\rm Spec}(G^{\prime})$, if $P^{\prime}$ contains  $(y_1^{\prime\prime},\ldots,y_l^{\prime\prime})$ and $G_{01}$, then $P^{\prime}$ contains $\sum_{\beta=n+1}^{n+l}w_{\alpha \beta}y_{\beta-n}^{\prime}$ for $1\leq \alpha \leq l$.
Since the matrix $(w_{\alpha\beta})$, $1\leq \alpha\leq l$, $n+1\leq \beta\leq n+l$, is invertible over $G^{\prime}$, $P^{\prime}$ contains $y_{\mu}^{\prime}$ for $1\leq \mu\leq l$. Thus $P^{\prime}\supseteq I$. This shows that $y_1^{\prime\prime},\ldots,y_l^{\prime\prime},x_1^{\prime},\ldots,x_s^{\prime}$
form a filter-regular sequence for $G^{\prime}$ with respect to $I$. By a similar argument, one can actually show that they  form a super-reduction for $I$. Let
$\mathfrak{P}^{\prime}=\{P^{\prime}\in {\rm
Min}(y_1^{\prime\prime},\ldots,y_l^{\prime\prime},x_1^{\prime},\ldots,x_{s-1}^{\prime})\,|\,
{\rm dim}\,G^{\prime}/P^{\prime}=1, G_{01}G^{\prime}\nsubseteq
P^{\prime}\}$. Notice that $y_1^{\prime\prime},\ldots,y_l^{\prime\prime}$ are generic linear combinations of the generators of $(G_{10}, G_{01})G^{\prime}$. By the proof of Theorems 3.1, 4.1 and the argument similar to \cite[6.5]{SUV}, we have
\begin{align}
& re(A^{\prime\prime})= j_d(I)\notag\\
=& \,\lambda(G^{\prime}/((y_1^{\prime\prime},\ldots,y_l^{\prime\prime},x_1^{\prime},
\ldots,x_{s-1}^{\prime})G^{\prime}:_{G^{\prime}}\langle I\rangle+x_{s}^{\prime}G^{\prime}))\notag\\
=& \,e(G^{\prime}/(y_1^{\prime\prime},\ldots,y_l^{\prime\prime},x_1^{\prime},
\ldots,x_{s-1}^{\prime})G^{\prime}:_{G^{\prime}}\langle I\rangle)\notag\\
=& \,\sum_{P^{\prime}\in \mathfrak{P}^{\prime}}
\lambda(G^{\prime}_{P^{\prime}}/(y_1^{\prime\prime},\ldots,y_l^{\prime\prime},x_1^{\prime},
\ldots,x_{s-1}^{\prime})G^{\prime}_{P^{\prime}})\,e(G^{\prime}/P^{\prime})\notag\\
=& \,e(G^{\prime}/((x_1^{\prime},\ldots,x_{s-1}^{\prime})G^{\prime}:_{G^{\prime}}\langle
G_{01}G^{\prime
}\rangle+(y_1^{\prime\prime},\ldots,y_l^{\prime\prime})))\notag\\
=& \,e(G^{\prime}/(x_1^{\prime},\ldots,x_{s-1}^{\prime})G^{\prime}:_{G^{\prime}}\langle
G_{01}G^{\prime}\rangle)\notag\\
=& \,e(G^{\prime}/((x_1^{\prime},
\ldots,x_{s-1}^{\prime})G^{\prime}:_{G^{\prime}}\langle
G_{01}G^{\prime}\rangle+x_{s}^{\prime}G^{\prime}))\notag\\
=& \,{\rm deg\,}\upsilon_s(x_1^{\prime},\ldots,x_s^{\prime},G^{\prime})\notag\\
=& \,c_{d-s}(A_1B^{\prime})=c_{d-s}(A_1B).\notag
\end{align}

Finally since $e(A^{\prime\prime})=e(A^{\prime})=e(A)$,  we are done.
\end{proof}

\begin{rema}
\rm   In Proposition 4.6, if in addition we assume that $B$ is equidimensional and universally catenary, then  $G$ is also equidimensional by \cite[2.2]{SUV}. Observe for every $p\in \mathfrak{P}$, there exists  $P\in {\rm Min}(G)$ which contracts back to $p$. Since $G$ is equidimensional, for every such $P$, ${\rm dim}\,G/P=d$. Notice ${\rm dim}\,A/p=s$. Hence ${\rm Quot}(G/P)$ has transcendence degree $l=d-s$ over ${\rm Quot}(A/p)$.  Therefore  ${\rm dim}\,A_p\otimes_A G=l$ for every prime ideal $p\in \mathfrak{P}$.
\end{rema}

\begin{rema}
\rm First observe in Proposition 4.6,  for each $p\in \mathfrak{P}$, since $A_p\otimes_A G$ is standard graded over $A_p$, $r={\rm deg}_{A_p}(A_p\otimes_A G)=e(A_p\otimes_A G)$ (see the argument before Proposition 4.6).
 Moreover if we assume $A$ is a domain, one can replace $r$ by ${\rm
rank}_{A[y_1,\ldots,y_l]}\, B$, where $y_{i}=\sum_{j=1}^{\tau}
\lambda_{ij}b_{j}$, $1\leq i\leq l$, for general elements
$(\lambda_{ij})\in k^{l\tau}$, and $k$ is the residue field of $A$.
\end{rema}

\begin{proof}
 One can show $({\rm
rank}_{A^{\prime}[y_1^{\prime},\ldots,y_l^{\prime}]}\,B^{\prime})\cdot
e(A)=c_{d-s}(A_1B)$ by  the same argument as in Proposition 4.6,
where $A^{\prime}$, $B^{\prime}$, $y_i^{\prime}$, $1\leq i\leq l$,
are constructed in the same way as in that proof. By
\cite[p.251]{SUV},  ${\rm
rank}_{A^{\prime}[y_1^{\prime},\ldots,y_l^{\prime}]}\,B^{\prime}={\rm
rank}_{A[y_1,\ldots,y_l]}\, B$ for general elements
$(\lambda_{ij})\in k^{l\tau}$. Since $r
e(A)=c_{d-s}(A_1B)$, we are done.
\end{proof}

\bigskip

By Theorem 3.1 and Proposition 4.6, we have the following theorem.

\begin{theo}
In the same setting as Proposition 4.6, let $g={\rm ht}\,A_1B$. For general elements $\Lambda=(\lambda_{ij})\in
R_0^{s n}$, define $ x_i$ and $  \mathfrak{Q}_i$, $0 \leq i \leq
s$, as before.  Then
\begin{align}
e(B)=re(A)+
 \sum_{q\in \mathfrak{Q}_0}e_{A_1B_q}(B_q)\,e(B/q)+
 \sum_{i={\rm max}\{1,\,g \}}^{s-1}\sum_{ q\in \mathfrak{Q}_i
 }j_{1}(\frac{A_1B_{q}}{(x_1,\ldots,x_{i-1})B_{q}})\,e(B/q).\notag
\end{align}

\end{theo}

\bigskip

By Corollaries 3.3, 3.5 and Proposition 4.6, we have the following
corollary.

\begin{coro}
Let $A\subseteq B$ be as in Proposition 4.6. Assume  $B$ is  Cohen-Macaulay.
Also write
$A_1A=(a_1,\ldots,a_n)A$, where $a_1,\ldots,a_n$ are homogeneous
elements of degree one, and  $g={\rm ht}\,A_1B$. Let
$\mathfrak{P}=\{p\in {\rm Spec}(A)\,|\, {\rm dim}\, A/p={\rm dim}\,A\}$ and
assume that $A_p$ is reduced and ${\rm deg}_{A_p}(A_p\otimes_A G)=r$ for
every prime ideal $p\in \mathfrak{P}$, where $G={\rm gr}_{A_1B}(B)$.
   For general elements $\Lambda=(\lambda_{ij})\in
R_0^{s n}$, define $ x_i$, $ H_{i-1}$ and $\Lambda_{i-1}$, $1 \leq i
\leq s$, as before.

$($a$)$
Assume the ideal $A_1B$ satisfies condition $G_s$ and Artin-Nagata property $AN^-_{s-2}$. Then
\begin{align}
e(B)=re(A)+ e(B/A_1B)+
 \sum_{i=g+1}^{s-1}e(B/(H_{i-1}+A_1B)).\notag
\end{align}

$($b$)$
 Assume $A_1B_q$ is a complete intersection for every $q\in
{\rm Proj}\, (B)$. Then
\begin{align}
e(B)=re(A)+e(B/A_1B)+
 \sum_{i=g+1}^{s-1}e(B/(F_{i-1}+A_1B)).\notag
\end{align}
where $F_{i-1}={\rm
 Fitt}_0(A_1B/(x_1,\ldots,x_{i-1})B)$.

$($c$)$ Assume  $A_1B$ is a perfect ideal of height $2$ which
satisfies condition $G_s$. Write $\mu(A_1B)=n$ and let $X_{n\times
(n-1)}$ be a matrix such that $A_1B=I_{n-1}(X)$.   Then
\begin{align}
e(B)=re(A)+e(B/A_1B)+
 \sum_{i=3}^{s-1}e(B/(I_n(X \,|\, \Lambda_{i-1
 }^T)+A_1B)).\notag
\end{align}

$($d$)$
Assume $A_1B$ is a perfect Gorenstein ideal of height $3$
which satisfies condition $G_{s}$. Write $\mu(A_1B)=n$ and let
$X_{n\times n}=(x_{ij})$ be an alternating matrix such that
$A_1B={\rm Pf}_{n-1}(X)$.  Let
 $T_{i-1}=\begin{pmatrix} X & \Lambda_{i-1}^T\\
 -\Lambda_{i-1} & 0 \end{pmatrix}$ and
 $J(T_{i-1})$ be the $B$-ideal generated by the Pfaffians of all
 principal sub-matrices of  $T_{i-1}$ which contain $X$ for $4\leq i\leq s-1$.
  Then
\begin{align}
e(B)=re(A)+e(B/A_1B)+
 \sum_{i=4}^{s-1}e(B/(J(T_{i-1})+A_1B)).\notag
\end{align}
\end{coro}

\begin{proof}
Since $B$ is Cohen-Macaulay, it is equidimensional and universally catenary. By Remark 4.7, ${\rm dim}\,A_p\otimes_A G=d-s$ for every prime ideal $p\in \mathfrak{P}$. We are done by Corollaries 3.3, 3.5 and Proposition 4.6.
\end{proof}

\bigskip

\section{Applications to the special fiber ring.}

\noindent

In this section, we are going to apply our formulas to
the special fiber ring $k\otimes_R \mathscr{R} (I)$, where
$I\subseteq R$ is an ideal generated by forms of the same degree
in a standard graded $k$-algebra $R$. We have the following
theorem:

\begin{theo}
Let $R=k[R_1]$ be a reduced standard graded Noetherian ring of
dimension $d$ over an infinite  field $k$. Let
$I=(a_1,\ldots,a_n)R$ be an $R$-ideal of height $g>0$, where $a_1,
\ldots, a_n$ are homogeneous elements of degree $\delta >0$. Write
$s=\ell(I)$, $k[I_{\delta}]$  the $k$-algebra generated by the forms in $I$ of degree $\delta$, $R^{(\delta)}$  the $\delta$-th Veronese subring of $R$ and $G={\rm
gr}_{I_{\delta}R^{(\delta)}}(R^{(\delta)})$.  Let $\mathfrak{P}=\{p\in {\rm Spec}(k[I_{\delta}]) \,|\, {\rm
dim} \,k[I_{\delta}]/p={\rm dim} \,k[I_{\delta}]\}$. For every
prime ideal $p\in \mathfrak{P}$, assume ${\rm dim}\,k[I_{\delta}]_p\otimes_{k[I_{\delta}]}G=d-s$  and  ${\rm deg}_{k[I_{\delta}]_p}(k[I_{\delta}]_p\otimes_{k[I_{\delta}]}G)=r$ $($notice when $s=d$, $r={\rm rank}_{k[I_{\delta}]_p}(R^{(\delta)})_p$$)$.    For general elements $\Lambda=(\lambda_{ij})\in k^{s
n}$, let $ x_i=\sum_{j=1}^n \lambda_{ij}a_j$, $J_{i-1}=(x_1, \ldots, x_{i-1})R:_R \langle I \rangle$ and $\mathfrak{Q}_i=\{q\in {\rm Min}(J_{i-1}+I)\,|\, {\rm dim}\,R/q=d-i\}$ for $1 \leq i \leq s$.
 Then
\begin{align}
e(k[I_{\delta}])=e(R)\frac{\delta^{d-1}}{r}-
 \sum_{i=g}^{s-1}\sum_{q\in \mathfrak{Q}_i}
 j_{1}(\frac{I_{q}}{(x_1,\ldots,x_{i-1})_{q}})\,e(R/q)\frac{\delta^{d-i-1}}{r}.
\end{align}
\end{theo}

\begin{proof}
  Let $ J_{i-1}^{(\delta)} =(x_1,\ldots,x_{i-1}) R^{(\delta)}:_{R^{(\delta)}}
\langle I_{\delta}R^{(\delta) }\rangle,\,\,1\leq i\leq s. $ For every $q \in
{\rm Proj} \,(R)$, let $q^{(\delta)}=q\cap R^{(\delta)}$. Then
$R^{(\delta)}_{q^{(\delta)}}\subseteq R_q$ is an $\acute{{\rm e}}$tale local
extension \cite{SUV}. By flatness we have
$J_{i-1}^{(\delta)}R_q=J_{i-1}R_q$.  Since there are natural isomorphisms
from ${\rm Proj}(R)$ to ${\rm Proj}(R^{(\delta)})$ and from ${\rm
Proj}(R/I)$ to ${\rm Proj}((R/I)^{(\delta)})={\rm
Proj}(R^{(\delta)}/I_{\delta}R^{(\delta)})$, they induce a one to one
correspondence between the set $\mathfrak{Q}_i$ and the set of all
primes $q^{(\delta)}$ in $V(J_{i-1}^{(\delta)}+I_{\delta}R^{(\delta)})$ with dim
$R^{(\delta)}/q^{(\delta)}=d-i$ for each $i$ with $1\leq i\leq s$. Let
$\overline{R}^{(\delta)}=R^{(\delta)}/(x_1,\ldots,x_{i-1})R^{(\delta)}$,
$\overline{R}=R/(x_1,\ldots,x_{i-1})R$.
 Since
$\overline{q}^{(\delta)}\overline{R}_{\overline{q}}=\overline{q}\overline{R}_{\overline{q}}$,
by flatness we have
$$\lambda_{\overline{R}^{(\delta)}_{\overline{q}^{(\delta)}}}
(\Gamma_{\overline{q}^{(\delta)}\overline{R}^{(\delta)}_{\overline{q}^{(\delta)}}}
(\overline{I}_{\delta}^{j}\overline{R}^{(\delta)}_{\overline{q}^{(\delta)}}/
\overline{I}_{\delta}^{j+1}\overline{R}^{(\delta)}_{\overline{q}^{(\delta)}}))
=\lambda_{\overline{R}_{\overline{q}}}
(\Gamma_{\overline{q}\overline{R}_{\overline{q}}}
(\overline{I}^j\overline{R}_{\overline{q}}/\overline{I}^{j+1}\overline{R}_{\overline{q}}))
$$
for every $j\geq 0$. Thus we have
$j_{1}(I_{\delta}\overline{R}^{(\delta)}_{\overline{q}^{(\delta)}})
=j_{1}(I\overline{R}_{\overline{q}})$. Also recall
$e(R^{(\delta)})=e(R)\,\delta^{d-1}$ and
$e(R^{(\delta)}/(J^{(\delta)}_{i-1}+I_{\delta}R^{(\delta)}))=
e(R/(J_{i-1}+I))\,\delta^{d-i-1}$.

We rescale the grading of both $k[I_{\delta}]$ and $R^{(\delta)}$ so that
$k[I_{\delta}]\subseteq R^{(\delta)}$ is a homogeneous inclusion of standard
graded Noetherian algebras. Notice $k[I_{\delta}]\simeq k\otimes_R
\mathscr{R}(I)$,  dim $k[I_{\delta}]=\ell (I)=s$ and ${\rm dim}\, R^{(\delta)}=d$.
Thus we apply Theorems 4.1 and   4.9 to the homogeneous
inclusion $k[I_{\delta}]\subseteq R^{(\delta)}$ to get Equation (7).
\end{proof}

\bigskip

By Theorem 5.1 and the proof of Corollaries 3.3 and 3.5, we have the following corollary:

\begin{coro}

Let $R$ and $I$ be as in Theorem 5.1. Assume  $R$ is
Cohen-Macaulay.  Also, let $\mathfrak{P}=\{p\in {\rm Spec}(k[I_{\delta}]) \,|\, {\rm
dim} \,k[I_{\delta}]/p={\rm dim} \,k[I_{\delta}]\}$.      For every
prime ideal $p\in \mathfrak{P}$, assume  ${\rm deg}_{k[I_{\delta}]_p}(k[I_{\delta}]_p\otimes_{k[I_{\delta}]}G)=r$ $($notice when $s=d$, $r={\rm rank}_{k[I_{\delta}]_p}(R^{(\delta)})_p$$)$, where $k[I_{\delta}]$, $R^{(\delta)}$ and $G$ are defined as in Theorem 5.1. For general elements $\Lambda=(\lambda_{ij})\in k^{s
n}$, let $ x_i=\sum_{j=1}^n \lambda_{ij}a_j$, $ H_{i-1}=(x_1, \ldots,x_{i-1})R:_R I$ and  $\Lambda_{i-1}$ be the sub-matrix consisting of the first $i-1$ rows of $\Lambda$ for $1 \leq i \leq s$.

$($a$)$ Assume the ideal $I$ satisfies condition $G_s$ and Artin-Nagata property $AN^-_{s-2}$. Then
\begin{align}
 e(k[I_{\delta}])=e(R)\frac{\delta^{d-1}}{r}-e(R/I)\frac{\delta^{d-g-1}}{r}-
 \sum_{i=g+1}^{s-1}e(R/(H_{i-1}+I))\frac{\delta^{d-i-1}}{r}.\notag
\end{align}

$($b$)$ Assume $I_q$ is a complete intersection for every $q\in
{\rm Proj}\, (R)$. Then
\begin{align}
 e(k[I_{\delta}])=e(R)\frac{\delta^{d-1}}{r}-e(R/I)\frac{\delta^{d-g-1}}{r}-
 \sum_{i=g+1}^{s-1}e(R/(F_{i-1}+I))\frac{\delta^{d-i-1}}{r}.\notag
\end{align}
where $F_{i-1}={\rm
 Fitt}_0(I/(x_1,\ldots,x_{i-1})R)$.

$($c$)$ Assume  $I$ is a perfect ideal of height $2$ which
satisfies condition $G_s$. Write $\mu(I)=n$ and let $X_{n\times
(n-1)}$ be a matrix such that $I=I_{n-1}(X)$.   Then
\begin{align}
 e(k[I_{\delta}])=e(R)\frac{\delta^{d-1}}{r}-e(R/I)\frac{\delta^{d-3}}{r}-
 \sum_{i=3}^{s-1}e(R/(I_n(X \,|\, \Lambda_{
 i-1}^T)+I))\frac{\delta^{d-i
 -1}}{r}.\notag
\end{align}

$($d$)$Assume $I$ is a perfect Gorenstein ideal of height $3$
which satisfies condition $G_{s}$. Write $\mu(I)=n$ and let
$X_{n\times n}$ be an alternating matrix such that
$I={\rm Pf}_{n-1}(X)$. Let
 $T_{i-1}=\begin{pmatrix} X & \Lambda_{i-1}^T\\
 -\Lambda_{i-1} & 0 \end{pmatrix}$ and
 $J(T_{i-1})$ be the $R$-ideal generated by the Pfaffians of all
 principal sub-matrices of  $T_{i-1}$ which contain $X$ for $4\leq i\leq s-1$.
   Then
\begin{align}
e(k[I_{\delta}])=e(R)\frac{\delta^{d-1}}{r}-e(R/I)\frac{\delta^{d-4}}{r}-
\sum_{i=4}^{s-1}e(R/(J(T_{i-1})+I))\frac{\delta^{d-i-1}}{r}.\notag
\end{align}
\end{coro}

\bigskip

Now we  are going to apply Theorem 5.1 and Corollary 5.2 to the following cases.

\bigskip

{\it Application 1.} If ${\rm char}\,k=0$ and $R$ is a domain,
 Theorem 5.1 yields an upper bound for the reduction number
$r_J(I)$ of $I$ with respect to any reduction $J$, because
$r_J(I)\leq e(k[I_{\delta}])$ by \cite{VA}. This gives a sharper bound
than that of \cite{SUV}.

{\it Application 2.} The ideal $I=(a_1,\ldots,a_n)R$ of Theorem
5.1 induces a rational map $ \phi: {\rm Proj}(R)\dashrightarrow
\mathbb{P}^{n-1}_k $ defined by $\phi(p)=(a_1(p),\ldots,a_n(p))$,
where $p\in {\rm Proj}(R)\backslash V(I)$ and $a_i(p)$ is $a_i$
evaluated at  $p$. Let ${\rm im}(\phi)$ be the image of $\phi$.
Then Theorem 5.1 provides a formula for the degree of ${\rm im}(\phi)$.
Indeed the homogeneous coordinate ring of ${\rm im}(\phi)$ is the
special fiber ring $k[I_{\delta}]$. Thus ${\rm deg}\,({\rm
im}(\phi))=e(k[I_{\delta}])$.

{\it Application 3} (Generalized Teissier's Pl$\ddot{{\rm
u}}$cker formula). Let $X\subseteq \mathbb{P}^n_k$ be a
hypersurface defined by a homogeneous irreducible polynomial $f$
of degree $\delta>0$ over an algebraically closed field $k$. Let
$X^{\prime}$ be the dual variety of $X$. Notice that in this case
$X^{\prime}$ is  the image of the Gauss map of $X$.  Set
$R=k[Y_0,\ldots,Y_n]/(f)$ and $I=(\partial f/\partial Y_0,
\ldots,\partial f/\partial Y_n)R$ (the Jacobian ideal of $R$).
Then  ${\rm dim}\,R=n$ and $I$ is generated by homogeneous elements of
degree $\delta-1$.  Since
$A(X^{\prime})\simeq k[I_{\delta-1}]$, we can apply Theorem 5.1 and Corollary 5.2 to
$R$ to get the degree $\delta^{\prime}$ of $X^{\prime}$.

Let $s=\ell(I)$ and $g={\rm ht}\,I$.
For general elements
$(\lambda_{ij})\in k^{s (n+1)}$,  let $ x_i=\sum_{j=0}^n
\lambda_{ij}\,\partial f/\partial Y_{j}$  and
$I_{i-1}=(x_1,\ldots,x_{i-1})R$ for $1\leq i \leq s$. Now for each
$i$, decompose $V(I_{i-1})$ into a union of irreducible components
$V_{i-1}^1,\ldots,  V_{i-1}^{b}$, where $V_{i-1}^j \nsubseteq {\rm
Sing\, } (X)$ for $1\leq j \leq a$ and $V_{i-1}^j \subseteq {\rm
Sing}\, (X)$ for $a+1\leq j \leq b$. Define $\mathfrak{Q}_i$ to be
the set of irreducible components of $(V_{i-1}^1\cup \cdots \cup
V_{i-1}^a )\cap {\rm Sing } \,(X)$ with dimension $n-i-1$.

Let $K={\rm Quot}(k[I_{\delta-1}])$, $G={\rm gr}_{(I_{\delta-1})R^{(\delta-1)}}(R^{(\delta-1)})$ and
$r={\rm deg}_K (K\otimes_{k[I_{\delta-1}]}G)$. Observe that if the dual variety
$X^{\prime}$ is again a hypersurface, i.e., $\ell (I)={\rm
dim}\,R=n$,  $r$ is just the degree of the Gauss map.

We have
\begin{align}
\delta^{\prime}=\frac{\delta(\delta-1)^{n-1}}{r}-
 \sum_{i=g}^{s-1}\sum_{q\in \mathfrak{Q}_i}
 j_{1}\big(I_q/(I_{i-1})_q)\,{\rm deg}
(q)\frac{(\delta-1)^{n-i-1}}{r}.\notag
\end{align}

This formula holds for a hypersurface with arbitrary dual variety and singularities. It
also gives a generalization of Teissier's Pl$\ddot{{\rm u}}$cker
formula.

Moreover,  letting $H_{i-1}=(x_1, \ldots, x_{i-1})R:_R I$ and $\Lambda_{i-1}$  be the sub-matrix consisting of
the first $i-1$ rows of $\Lambda=(\lambda_{ij})$ for $1\leq i \leq
s$,  we have the following cases:

$($a$)$
Assume   $I$   satisfies condition $G_s$ and Artin-Nagata property $AN^-_{s-2}$.  Then
\begin{align}
 \delta^{\prime}=&\frac{\delta(\delta-1)^{n-1}}{r}
 -{\rm deg}(V(I))\frac{(\delta-1)^{n-g-1}}{r}-\notag\\
 &-
 \sum_{i=g+1}^{s-1}{\rm deg}(V(H_{i-1}+I))\frac{(\delta-1)^{n-i-1}}{r}.\notag
\end{align}

$($b$)$ Assume $I$ is a complete intersection on $X$.   Then
\begin{align}
 \delta^{\prime}=&\frac{\delta(\delta-1)^{n-1}}{r}-{\rm deg}(V(I))\frac{(\delta-1)^{n-g-1}}{r}-\notag\\&-
 \sum_{i=g+1}^{s-1}{\rm deg}(V(F_{i-1} +I))\frac{(\delta-1)^{n-i-1}}{r}.\notag
\end{align}
where $F_{i-1}={\rm
 Fitt}_0(I/I_{i-1})$.

$($c$)$ Assume $I$ is a perfect ideal of height $2$ which
satisfies condition $G_s$. Write $\mu(I)=n$ and let $W_{n\times
(n-1)}$ be a matrix such that $I=I_{n-1}(W)$.  Then
\begin{align}
 \delta^{\prime}=&\frac{\delta(\delta-1)^{n-1}}{r}-{\rm deg}(V(I))\frac{(\delta-1)^{n-3}}{r}-\notag\\&-
 \sum_{i=3}^{s-1}{\rm deg}(V(I_n(W \,|\, \Lambda_{
 i-1}^T)+I)) \frac{(\delta-1)^{n-i-1}}{r}.\notag
\end{align}

$($d$)$ Assume $I$ is a perfect Gorenstein ideal of height $3$
which satisfies condition $G_s$.
  Write $\mu(I)=n$ and
  let $W_{n\times n}$ be an alternating matrix such that $I={\rm
Pf}_{n-1}(W)$. Let
 $T_{i-1}=\begin{pmatrix} W & \Lambda_{i-1}^T\\
 -\Lambda_{i-1} & 0 \end{pmatrix}$ and
 $J(T_{i-1})$ be the $R$-ideal generated by the Pfaffians of all
 principal sub-matrices of  $T_{i-1}$ which contain $W$.
Then
\begin{align}
\delta^{\prime}=&\frac{\delta(\delta-1)^{n-1}}{r}-{\rm
deg}(V(I))\frac{(\delta-1)^{n-4}}{r}-\notag\\&-
 \sum_{i=4}^{s-1}{\rm deg}(V(J(T_{i-1})+I))\frac{(\delta-1)^{n-i-1}}{r}.\notag
\end{align}

\bigskip

We will finish the paper by giving the following  example.
\medskip

Example: Surfaces in 3-Space (see \cite[9.3.7]{F} or \cite{TH})
\medskip

Let $X: y_1^2y_3-y_2^2y_0=0$ be a hypersurface of degree 3 in $\mathbb{P}_k^3$, where $k$ is an algebraically closed field.
Observe $R=k[y_0, y_1, y_2, y_3]/(y_1^2y_3-y_2^2y_0)=k[\overline{y}_0, \overline{y}_1, \overline{y}_2, \overline{y}_3]$ is the homogeneous coordinate ring of $X$ and  $I=V(\overline{y}_2^2, \overline{y}_1\overline{y}_3, \overline{y}_2\overline{y}_0, \overline{y}_1^2)$ is the Jacobian ideal of $R$.
The singularity locus ${\rm Sing}\,(X)$ of $X$ is equal to
$V(I)$ which is a line  consisting of double points and 2 pinch points: $(0,0,0,1), (1,0,0,0)$.
Since the special fiber ring $F(I)=k[\overline{y}_2^2, \overline{y}_1\overline{y}_3, \overline{y}_2\overline{y}_0, \overline{y}_1^2]$ is the homogeneous coordinate ring of the dual variety $X^{\prime}$,  we can use our formula to find the class $\delta^{\prime}={\rm degree}(X^{\prime})$ of the variety $X$.

Observe ${\rm dim}\, R=3$, $I$ is generated by forms of the same degree $2$, ${\rm ht}\,I=1$, $\ell(I)=3$ and  $r=1$ the degree of the Gauss map.
Let $x_1=\overline{y}_1\overline{y}_3-\overline{y}_2\overline{y}_0$, $x_2=\overline{y}_1^2-\overline{y}_2^2+\overline{y}_1\overline{y}_3$, then
$$\mathfrak{Q}_1=\{(\overline{y}_1, \overline{y}_2)\}, \,\,\mathfrak{Q}_2=\{(\overline{y}_0, \overline{y}_1, \overline{y}_2), (\overline{y}_1, \overline{y}_2, \overline{y}_3), (\overline{y}_0-\overline{y}_3, \overline{y}_1, \overline{y}_2)\},$$
and
$$\delta^{\prime}={\rm degree}(X^{\prime})=e(F(I))$$
$$=3\cdot 2^2-2\cdot \sum_{q\in \mathfrak{Q}_1} j_1(I_q)- \sum_{q\in \mathfrak{Q}_2} j_1(I_q/(\overline{y}_1\overline{y}_3-\overline{y}_2\overline{y}_0)_q)$$
$$=3\cdot 2^2-2\cdot 2-(2+2+1)=3.$$

Indeed, $X^{\prime}$ is a hypersurface which is isomorphic to $X$.

\bigskip

\begin{center}
Acknowledgements
\end{center}

I am very grateful to the referee for a meticulous reading and
 detailed recommendations which led to the improvements
in the text.

\end{document}